\documentclass[11pt,reqno]{amsart}
\usepackage[latin9]{inputenc}
\usepackage{color}
\usepackage{cancel}
\usepackage{amsbsy}
\usepackage{amstext}
\usepackage{amsthm}
\usepackage{amssymb}
\usepackage{mathtools}
\PassOptionsToPackage{normalem}{ulem}
\usepackage{ulem}
\usepackage[letterpaper,margin=1in]{geometry}
\usepackage[T1]{fontenc}
\usepackage{lmodern}
\usepackage{amscd}
\usepackage{tikz-cd}
\usepackage{microtype}
\usepackage{amsmath,amssymb,amsthm,mathtools,mathrsfs}
\usepackage{booktabs}
\usepackage{enumitem}
\usepackage{enumitem}
\usepackage{xcolor}
\usepackage{hyperref}

\makeatletter
\numberwithin{equation}{section}
\numberwithin{figure}{section}

\usepackage{tikz-cd}
\usepackage{fancyhdr}
\usepackage{latexsym}
\usepackage{amscd}\usepackage{amsfonts}\usepackage{pstricks}
\usepackage{young}
\usepackage{enumitem}
\usepackage{amsthm}
\usepackage{mathrsfs}
\usepackage{enumitem}

\usepackage{ifthen}
\usepackage{longtable}
\usepackage{todonotes}
\usepackage{marginnote}

\newcommand{\Lie}{\operatorname{Lie}}

\newcommand{\St}{\operatorname{St}}

\newcommand{\codim}{\operatorname{codim}}
\newcommand{\ord}{\operatorname{ord}}

\newcommand{\id}{\operatorname{id}}
\newcommand{\reg}{\operatorname{reg}}
\newcommand{\sing}{\operatorname{sing}}

\newcommand{\Zar}{\operatorname{Zar}}

\newcommand{\eps}{\varepsilon}

\renewcommand{\subsection}[1]{\vspace{3mm}\refstepcounter{subsection}\noindent{\bf \thesubsection. #1.} }

\renewcommand{\subsubsection}[1]{\vspace{3mm}\refstepcounter{subsubsection}\noindent{\bf \thesubsubsection. #1.} }

\numberwithin{equation}{section}

\newtheorem{theorem}{Theorem}[section]
\newtheorem{theoremA}{Theorem}

\newtheorem{proposition}[theorem]{Proposition}
\newtheorem{lemma}[theorem]{Lemma}
\newtheorem{sublemmam}[theorem]{Main Sub-Lemma}
\newtheorem{sublemma}[theorem]{Sub-Lemma}

\theoremstyle{definition}

\theoremstyle{remark}

\makeatother

\begin{document}
\title[Holomorphic Maps from $\mathbb{C}^p$ into Semi-Abelian Varieties ]{Holomorphic Maps from \(\mathbb{C}^p\) into Semi-Abelian Varieties}

\author{Zhe Wang}
\address{
 Department  of Mathematics\newline
\indent University of Houston\newline
\indent Houston,  TX 77204, U.S.A.} 
\email{zwang224@cougarnet.uh.edu}
\thanks{2020\ {\it Mathematics Subject Classification}: 32H30, 32A22, 32Q45.}
\thanks{}

\begin{abstract} We prove the Second Main Theorem with truncation level
one for holomorphic maps with Zariski-dense image intersecting a reduced effective Cartier divisor: Let $A$ be a semi-abelian variety with an equivariant compactification $A\subset\overline A$, and let $f:\mathbb C^p\to A$ be a holomorphic map with Zariski-dense image. Let $D$ be a reduced effective divisor on $A$ extending to a divisor $\overline D$ on $\overline A$. After possibly replacing $\overline A$ by another equivariant compactification depending only on $D$ and independent of $f$, for every $\varepsilon>0$,
\[
T_f(r,\overline D)\leq_{\mathrm{exc}}N_f^{[1]}(r,D)+\varepsilon T_f(r,\overline D).
\]
 \end{abstract}

\maketitle
\baselineskip=16truept

 \section{Introduction}\label{sec:intro}
In \cite{WangCp}, the author used Stoll's fiber integration method to extend the result of Noguchi--Winkelmann--Yamanoi on the Second Main Theorem with finite truncation from holomorphic curves to holomorphic maps from \(\mathbb{C}^p\). We recall this result below. For notation, see \cite{WangCp}, \cite{CaiRuYang2},  \cite{book}, or \cite{NWbook}.

\begin{theoremA}\label{SEMIABELIAN} 
Let $A$ be a semi-abelian variety with an equivariant compactification $A\subset\overline A$, and let $f:\mathbb C^p\to A$ be a holomorphic map with Zariski-dense image. Let $D$ be an effective divisor on $A$ extending to a divisor $\overline D$ on $\overline A$. After possibly replacing $\overline A$ by another equivariant compactification depending only on $D$ and independent of $f$, there exists a positive integer $\rho$, depending on $f$ and $D$, such that
\begin{equation}
T_f(r,\overline D)
\leq
N_f^{[\rho]}(r,D)+S_{f, \overline D}(r),
\end{equation}
where $S_{f, \overline D}(r)$ denotes a nonnegative small term such that $S_{f, \overline D}(r)\leq_{\mathrm{exc}} o(T_f(r, \overline D))$, where $``\leq_{\mathrm{exc}}"$ means that the inequality holds for all $r\in (0, \infty)$ except a set of finite Lebesgue measure.
\end{theoremA}

The purpose of the present paper is to address the more delicate problem of lowering the finite truncation level in Theorem \ref{SEMIABELIAN} to one, extending the result of Noguchi-Winkelmann-Yamanoi \cite{NWY2008}. Our main result is the following.

\begin{theorem}[Main Theorem]\label{thm:main}

Let $A$ be a semi-abelian variety with an equivariant compactification $A\subset\overline A$, and let $f:\mathbb C^p\to A$ be a holomorphic map with Zariski-dense image. Let $D$ be a reduced effective divisor on $A$ extending to a divisor $\overline D$ on $\overline A$. After possibly replacing $\overline A$ by another equivariant compactification depending only on $D$ and independent of $f$, for every $\varepsilon>0$,
\begin{equation}
T_f(r,\overline D)
\leq_{\mathrm{exc}}
N_f^{[1]}(r,D)+\varepsilon T_f(r,\overline D).
\end{equation}

\end{theorem}
In the case $p=1$, according to the arguments of Yamanoi \cite{Y04} (as well as Noguchi-Winkelmann-Yamanoi \cite{NWY2008}), reducing from $N_f^{[\rho]}(r,D)$ to $N_f^{[1]}(r,D)$  requires  estimates for the   truncation-one counting functions  for the  higher codimension subvarieties $Z\subset A$ and $Z \subset X_1\cap J_1(D)$, where $X_1=\overline{j_1f({\Bbb C})}^{\,\Zar}$ for a holomorphic map $f:\mathbb C\to A$. 
More precisely,  we need, fixing an ample divisor $H$ on $\overline{A}$, 

(i) For $Z\subset A$ with $\codim_A Z\ge 2$
and  every $\varepsilon>0$, $
N^{[1]}(r,f^*Z)
\leq_{\mathrm{exc}}
\varepsilon T_{f, H}(r)$. 

(ii)  For 
$Z\subset X_1\cap J_{1}(D)$ with $\codim_{X_1} Z\ge 2$ and every $\varepsilon>0$,  $N^{[1]}\!\left(r,(j_{1}f)^*Z\right)
\leq_{\mathrm{exc}}
\varepsilon T_{f, H}(r).$

The above  two estimates are called the {\it GCD estimates}. The key ingredient in the above GCD estimates is the following fact: if $B$ is a projective variety of dimension $b$ and $Z$ is a subvariety with $\codim_{B} Z\ge 2$, and $M$ is ample on $B$, then
$h^0(B, M^m)\asymp m^b$,  and $h^0(Z, M^m|_Z)=O( m^{b-2})$. 

To deal with the $p>1$ case,  the formally natural higher dimensional analogue of the $l$-jet of a holomorphic curve is the $l$-jet of \(p\)-germs. It records all partial derivatives $\partial^\alpha f$ with $|\alpha|\leq l$. The theory of the $p$-germ jets has already been introduced and discussed in the recent papers  (see Etesse \cite{Ete} and Cai-Ru-Yang \cite{CaiRuYang2}). 
However, a direct use of all partial derivatives of order at most $ l $ is not appropriate. The reason is as follows. In the one-variable case, the nonreduced contact scheme is thickened along the jet direction determined by the based jet, which varies with the base point. It gives contact multiplicity at least $ l +1$ and has generic rank $ l +1$ over its reduced support.
Constructing an auxiliary divisor of degree $n( l )$ therefore requires $ l /n( l )^2\to0$, while obtaining the smallness of the error term requires $n( l )/ l \to0$. Such a sequence exists, for example $n( l )=\lfloor l ^{3/4}\rfloor$. For the $p$-germ jet, however, the analogous construction would thicken the based jet along the $p$-germ jet direction. The contact scheme would then have generic rank $\binom{ l +p}{p}\asymp l ^p$ over its reduced support. Consequently, constructing an auxiliary divisor of degree $n( l )$ would require $ l ^p/n( l )^2\to0$. On the other hand, the divisorial multiplicity would remain only $ l +1$. Thus one would need $ l ^{p/2}\ll n( l )\ll l $, which is impossible for $p\ge2$.

To overcome this difficulty, we choose a generic constant direction. The key idea of our approach is to use two jets: the $p$-germ jets and the directional jet. 
In particular, when deriving the GCD estimates for higher codimensional subvarieties, we choose a generic constant direction
$\delta=\sum_{\nu=1}^p a_\nu\partial/\partial z_\nu$
and use the associated directional jet $j_k^\delta(f)=(f,\delta f,\ldots,\delta^k f).$
At a point $z\in\mathbb C^p$, this is the jet of the one-variable germ $t\mapsto f(z+ta)$. It is therefore adapted to the nonreduced contact scheme construction in Yamanoi's argument. The main new point is to prove that the direction $a$ can be chosen so that the directional order of vanishing along a general transverse complex line agrees with the divisorial order of vanishing on $\mathbb C^p$.

The $p$-germ jet nevertheless remains essential, since the logarithmic derivative Lemma for several variables controls all jet coordinates and hence the corresponding directional jet coordinates.

\section{$p$-Germ Jet Estimates and Smallness}

To extend the jet estimates of Noguchi--Winkelmann--Yamanoi
\cite{NWY2002} from holomorphic curves to holomorphic maps from
\(\mathbb{C}^p\), we apply the following logarithmic derivative estimate for jet
differentials of \(p\)-germs (see Lemma 3.2 in \cite{CaiRuYang2}). It yields an analogue in several
variables of their smallness theorem for jet coordinates
\cite[Lemma~3.8]{NWY2002} and, together with the jet projection method,
gives the quotient-smallness estimate needed below. For notation of $p$-germ jets and $p$-germ jet differentials, see \cite{CaiRuYang2}.

\begin{lemma}[Lemma 3.2 in \cite{CaiRuYang2}]\label{mainlemma}

 Let $X$ be a smooth projective variety of dimension $n$ and $D$ be a normal crossing divisor on $X$. Fix a multi-degree $\boldsymbol{\beta} = (\beta_1,\ldots, \beta_p) \in (\mathbb Z _{\geq 0})^p.$  Let  
	\begin{equation*}
		\mathcal P \in H^0( X,E_{p,k,\boldsymbol{\beta}}^{GG}(X, \log D ))
	\end{equation*}
	be a logarithmic $k$-jet ($k\geq 1$) differential of $p$-germs of  multi-degree $\boldsymbol{\beta}.$ Let $f: \mathbb{C}^p \rightarrow X$ be a meromorphic map such that $f(\mathbb{C}^p) \not\subset  \operatorname{Supp}(D)$. We have
	\begin{eqnarray*}
\int_{S_p(r)} \log^+\big|\mathcal P( j_{p,k}( f ) )(w)\big|  \sigma _p \leq S_{f, H}(r),	
\end{eqnarray*}
	where $H$ is any (fixed) ample divisor on $X$.
\end{lemma}

\begin{lemma}\label{lem:log-jet-coordinate-small}
Let $A$ be a semi-abelian variety and let
$
f:\mathbb C^p\longrightarrow A
$
be a holomorphic map. Fix a smooth equivariant compactification
$A\subset\overline A$ and an ample line bundle $H$ on $\overline A$.
For every fixed $k\geq1$, let
$
j_{p,k}(f)=\bigl(f,\xi_{p,k}(f)\bigr)
$
be the $p$-germ jet lifting, where
$\xi_{p,k}(f)$ denotes the jet coordinate part. Then
\[
T_{\xi_{p,k}(f)}(r)
\leq
S_{f, H}(r).
\]
\end{lemma}

\begin{proof}
Let
$\partial A=\overline A\setminus A,~n=\dim A,$
and let
$q:\mathbb C^n\longrightarrow A$
be the universal covering map. Choose linear coordinates
\(e_1,\ldots,e_n\) on \(\mathbb C^n\). Since \(\overline A\) is a
smooth equivariant compactification of \(A\), its logarithmic
cotangent bundle is trivial. We may therefore choose a global
invariant frame
$
\omega_1,\ldots,\omega_n
\in
H^0\!\left(
\overline A,\Omega_{\overline A}^1(\log\partial A)
\right)
$
such that
$
q^*\omega_\lambda=de_\lambda,
~ 1\leq\lambda\leq n.
$

Let
$
\bar u=(u_1,\ldots,u_m)\in\mathcal W_{p,\leq k}
$
be a nonempty word, and let $\boldsymbol{\alpha}(\bar u)=\bigl(\alpha_1(\bar u),\ldots,\alpha_p(\bar u)\bigr)\in(\mathbb Z_{\geq0})^p.$ For
\(1\leq\lambda\leq n\) and a \(p\)-germ
\(h:(\mathbb C^p,0)\to  A\), define
\[
\mathcal P_{\lambda,\bar u}(j_{p,m}(h))
=
\partial_{u_2}\cdots\partial_{u_m}
\left[
h^*\omega_\lambda
\left(\frac{\partial}{\partial z_{u_1}}\right)
\right](0).
\]

We have
\[
\mathcal P_{\lambda,\bar u}
\in
H^0\left(
\overline A,
E^{\mathrm{GG}}_{p,m,\boldsymbol{\alpha}(\bar u)}
(\overline A,\log\partial A)
\right),
\]
and the coordinate
functions of \(\xi_{p,k}(f)\) are precisely
$\mathcal P_{\lambda,\bar u}(j_{p,m}(f)),~1\leq\lambda\leq n,~1\leq|\bar u|\leq k.$

Fix \(\lambda\) and \(\bar u\). 
Lemma~\ref{mainlemma} applied to
\(\mathcal P_{\lambda,\bar u}\), gives
\[
\int_{S_p(r)}
\log^+
\left|
\mathcal P_{\lambda,\bar u}(j_{p,m}(f))(w)
\right|
\,\sigma_p
\leq
S_{f, H}(r).
\]
The function
\(\mathcal P_{\lambda,\bar u}(j_{p,m}(f))\) is holomorphic on \(\mathbb{C}^p\). Hence
\[
T_{\mathcal P_{\lambda,\bar u}(j_{p,m}(f))}(r)
=
\int_{S_p(r)}
\log^+
\left|
\mathcal P_{\lambda,\bar u}(j_{p,m}(f))(w)
\right|
\,\sigma_p
+O(1)
\leq
S_{f, H}(r).
\]

For fixed \(k\), the map \(\xi_{p,k}(f)\) has only finitely many
coordinate functions \(\mathcal P_{\lambda,\bar u}(j_{p,m}(f))\).  We have
\[
\begin{aligned}
T_{\xi_{p,k}(f)}(r)\leq S_{f, H}(r).
\end{aligned}
\]
This proves the lemma.
\end{proof}

\begin{proposition}\label{prop:qsmall}
Let \(A\) be a semi-abelian variety and let \(g=(f,\phi):\mathbb C^p\to A\times\mathbb A^N\) be a holomorphic map with \(T_{\phi}(r)\leq S_{f,H}(r)\), and set \(X=\overline{g(\mathbb C^p)}^{\,\Zar}\). Let \(B=\St_A^0(X)\), \(C=A/B\), \(q_B:A\to C\) be the quotient map, \(\pi_B=q_B\times\operatorname{id}_{\mathbb A^N}\), and \(Y=X/B\). Then, for every ample
line bundle \(M\) on a projective compactification of \(Y\),
\begin{equation}\label{eq:qsmall}
 T_{\pi_B\circ g}(r,M)
 \leq S_{f, H}(r).\nonumber
\end{equation}
\end{proposition}

\begin{proof}
Write
\[
 \pi_B\circ g=(q_B\circ f,\phi):\mathbb C^p\to Y\subset C\times \mathbb{A}^N.
\]
We have $\St_C^0(Y)=\{0\}$. For \(k\geq0\), let $Y_k=\overline{j_{p,k}(\pi_B\circ g)(\mathbb{C}^p)}^{\,\Zar}\subset C\times \mathbb{A}^N \times J_{p,k,C\times \mathbb{A}^N } $ and let $\chi_k:Y_k\rightarrow \mathbb{A}^N \times J_{p,k,C\times \mathbb{A}^N }$, where $J_{p,k,C\times\mathbb A^N}$ denotes the fiber over $(0_C,0)$ of $J_{p,k}(C\times\mathbb A^N)\to C\times\mathbb A^N$.
Note that $J_{p,k,C\times \mathbb A^N}
\simeq
\mathbb C^{
(\dim C+N)\left(\binom{p+k}{p}-1\right)}. $ 

Fix \(0\in\mathbb C^p\) and identify the fiber of $\chi_k$ over
$\chi_k(j_{p,k}(\pi_B\circ g)(0))$
with a closed subset of $C$.  Set
\[
 V_k=\left\{c\in C~|~
 (c,\chi_k(j_{p,k}(\pi_B\circ g)(0)))\in Y_k\right\}.
\]
Then \(V_k\) are nonempty closed algebraic subsets of \(C\), and satisfy \(V_{k+1}\subset V_k\). Thus we have the sequence of Zariski closed sets
$$
\cdots \subset V_3 \subset V_2 \subset V_1 
$$
that eventually stabilizes to a closed subset \(V\).

Let $c\in V$ and translate $\pi_B\circ g$ by
$ c-(q_B\circ f)(0)$ by putting $\widetilde{\pi_B\circ g}=\pi_B\circ g+c-(q_B\circ f)(0)$. Then its \(k\)-jet at
\(0\) lies in \(Y_k\) for every \(k\geq0\).  Since
\(Y_k\subset J_{p,k}Y\), we have $ j_{p,k}(\widetilde{\pi_B\circ g})(0)\in J_{p,k}Y$. It implies that 
\((\widetilde{\pi_B\circ g})(\mathbb{C}^p)\subset Y\).  Therefore \(c-(q_B\circ f)(0)\) stabilizes \(Y\).  The stabilizer $ \St_C(Y)$ is
finite, so  \(V\) is finite. It follows that, for some \(k_0\), \(\chi_{k_0}:\overline{Y_{k_0}}\dashrightarrow \overline{\chi_{k_0}(Y_{k_0})}\) is generically finite, where $\overline{Y_{k_0}}$ is a projective compactification of $Y_{k_0}$, and $\overline{\chi_{k_0}(Y_{k_0})}$ is the Zariski closure of $\chi_{k_0}(Y_{k_0})$ in a projective compactification of $\mathbb A^N\times J_{p,k_0,C\times\mathbb A^N}$.
\[
\begin{tikzcd}[
    column sep=6em,
    row sep=4em
]
& Y_{k_0}
    \arrow[r, "\chi_{k_0}"]
   \arrow[d, "\rho_{k_0,0}"']
&
\chi_{k_0}(Y_{k_0})\subset \mathbb{A}^N
\times J_{p,k_0,C\times\mathbb A^N}
\\
\mathbb{C}^p
    \arrow[ur, "j_{p,{k_0}}(\pi_B\circ g)"]
    \arrow[r, "\pi_B\circ g"']
&
Y
\end{tikzcd}
\]

Let $L$ be an ample line bundle on $\overline{\chi_{k_0}(Y_{k_0})}$. Since \(\chi_{k_0}:\overline{Y_{k_0}}\dashrightarrow \overline{\chi_{k_0}(Y_{k_0})}\) is generically finite, \cite[Proposition~2.5.20]{NWbook} gives 
\begin{equation}\label{eq:jet-generic-finite}
T_{j_{p,k_0}(\pi_B\circ g)}(r,M_{k_0})
=
O\!\left(
T_{\chi_{k_0}\circ j_{p,k_0}(\pi_B\circ g)}(r,L)
\right),
\end{equation}
where $M_{k_0}$ is any ample line bundle on $\overline{Y_{k_0}}$.

Now we claim that 
\[
T_{\chi_{k_0}\circ j_{p,{k_0}}(\pi_B\circ g)}(r, L)
\leq 
S_{f, H}(r).
\]

Indeed, we have the jet lift
$j_{p,{k_0}}(\pi_B\circ g)=\left(q_B\circ f,\,\phi,\,\xi_{p,{k_0}}(\pi_B\circ g)\right),$
where $\xi_{p,{k_0}}$ denotes the jet coordinates.
The map $\chi_{k_0}$ forgets the $C$-coordinate, it follows
$\chi_{k_0}\circ j_{p,{k_0}}(\pi_B\circ g)=\left(\phi,\,\xi_{p,{k_0}}(\pi_B\circ g)\right).$
By assumption, $T_\phi(r)\leq S_{f,H}(r)$, which, together with Lemma~\ref{lem:log-jet-coordinate-small}, implies $T_{\xi_{p,k_0}(\pi_B\circ g)}(r)\leq S_{f,H}(r)$.
Thus
\[
T_{\chi_{k_0}\circ j_{p,{k_0}}(\pi_B\circ g)}(r, L)
\leq S_{f, H}(r).
\]
This proves our claim.
By (\ref{eq:jet-generic-finite}), we have
\begin{equation}\label{eq:jet-smallness}
T_{j_{p,k_0}(\pi_B\circ g)}(r,M_{k_0})
\leq S_{f,H}(r).
\end{equation}
Next, we prove that
\[
T_{\pi_B\circ g}(r,M)
=
O\!\left(
T_{j_{p,k_0}(\pi_B\circ g)}(r,M_{k_0})
\right).
\]
Indeed, the natural projection $\rho_{k_0,0}:Y_{k_0}\longrightarrow Y$
satisfies 
$\rho_{k_0,0}\circ j_{p,k_0}(\pi_B\circ g)=\pi_B\circ g.$
Hence, by \cite[Proposition~2.5.20]{NWbook}, we have
\begin{equation}\label{eq:projection-characteristic}
T_{\pi_B\circ g}(r,M)
=
O\!\left(
T_{j_{p,k_0}(\pi_B\circ g)}(r,M_{k_0})
\right).
\end{equation}
Thus, by (\ref{eq:jet-smallness}) and (\ref{eq:projection-characteristic}), we obtain
\[
\begin{aligned}
T_{\pi_B\circ g}(r,M)
&=
O\!\left(
T_{j_{p,k_0}(\pi_B\circ g)}(r,M_{k_0})
\right)\\
&\leq
S_{f, H}(r).
\end{aligned}
\]
\end{proof}

\section{The directional jets}

\begin{lemma}\label{lem:generic}
Let $f:\mathbb C^p\to A$ be a holomorphic map with Zariski-dense image, and let \(\{E_\nu\}_{\nu\geq1}\) be a countable family of irreducible
analytic hypersurfaces in \(\mathbb C^p\).  There exists
$
 0\neq a=(a_1,\ldots,a_p)\in\mathbb C^p,~
 \delta_a=\sum_{j=1}^pa_j\frac{\partial}{\partial z_j},
$
such that:
\begin{enumerate}
\item \(\delta_a\) is transverse to every \(E_\nu\) at a general smooth
      point;
\item \(\delta_a f\not\equiv0\).
\end{enumerate}
\end{lemma}

\begin{proof}
For each $\nu\geq1$, choose a reduced entire function $h_\nu\in\mathcal O(\mathbb C^p)$ such that $E_\nu=\{h_\nu=0\}$. Put $L_\nu=\{a\in\mathbb C^p~|~\left.\sum_{j=1}^pa_j\partial_jh_\nu \right|_{E_\nu}\equiv0\}$. Then $L_\nu$ is a complex linear subspace of $\mathbb C^p$. Since $h_\nu$ is reduced, $dh_\nu$ does not vanish identically along $E_\nu$, and hence $L_\nu$ is proper. For every $a\notin L_\nu$, $\delta_a h_\nu$ does not vanish identically on $E_\nu$, so $\delta_a$ is transverse to $E_\nu$ at a general smooth point.

Put $M_f=\{a\in\mathbb C^p~|~\delta_a f\equiv0\}$. It is a proper complex linear subspace of $\mathbb C^p$, since otherwise $df\equiv0$ and hence $f$ would be constant, contradicting the Zariski density of its image. A countable union of proper complex linear subspaces cannot cover $\mathbb C^p$, proving the lemma.
\end{proof}

Let \(\mathbb{A}^N\) be an affine space with trivial \(A\)-action, and let
\[
 g=(f,\phi):\mathbb C^p\longrightarrow A\times \mathbb{A}^N
\]
be a holomorphic map with $T_{\phi}(r)\leq S_{f,H}(r)$.

For the direction fixed in Lemma~\ref{lem:generic}, and put
$\delta=\delta_a.$ Define
\begin{equation}\label{eq:mixedtower}
 j_l^{\delta}(g)(z)=j_l\bigl(t\longmapsto g(z+ta)\bigr)\big|_{t=0},
 \qquad
 X_l=\overline{j_l^{\delta}(g)(\mathbb C^p)}^{\,\Zar}\subset A\times \mathbb{A}^N\times J_{l,A\times\mathbb{A}^N},\nonumber
\end{equation}
where $J_{l,A\times\mathbb A^N}$ denotes the fiber over the identity $(0_A,0)$ of the projection $J_l(A\times\mathbb A^N)\to A\times\mathbb A^N$.

\begin{lemma}\label{cor:commonq}
Let $G=\cap_{l \geq 0} \mathrm{St}_A^0(X_l)$. For every
fixed \(k\), the quotient map
$
 \pi_k:X_k\longrightarrow Y_k=X_k/G
$
satisfies
\[
 T_{\pi_k\circ j_k^{\delta}(g)}(r,M_k)\leq S_{f, H}(r)
\]
for every line bundle \(M_k\) on a fixed projective compactification
of \(Y_k\).
\end{lemma}

\begin{proof}
Apply Proposition~\ref{prop:qsmall} to obtain
smallness of \(q_G\circ f\).  For any fixed \(k\), the coordinates of 
\(\pi_k\circ j_k^{\delta}(g)\) consist of \(q_G\circ f\) and
finitely many jet coordinates, all of which have
characteristic bounded by  \(S_{f, H}(r)\).
\end{proof}

\begin{lemma}\label{lem:nondominant}
Let \(Z\subset X_k\) be a closed subvariety.  If
\(\overline{\pi_k(Z)}\subsetneq Y_k\), then
\[
 N^{[1]}(r,j_k^{\delta}(g)^*Z)\leq S_{f, H}(r).
\]
\end{lemma}

\begin{proof}
Choose an effective divisor \(E\) on a projective compactification
of \(Y_k\) containing \(\overline{\pi_k(Z)}\). We have 
$$N^{[1]}(r,j_k^{\delta}(g)^*Z)\leq N(r,(\pi_k\circ j_k^{\delta}(g))^*E)\leq T_{\pi_k\circ j_k^{\delta}(g)}(r,E).$$
Lemma~\ref{cor:commonq} gives the assertion.
\end{proof}

We establish a higher codimension estimate for the directional jet lifting of a holomorphic map. This is a key technical ingredient in the proof of the main theorem. Our proof follows the higher codimensional argument of Noguchi--Winkelmann--Yamanoi; see Theorem~5.1 and Lemmas~5.20--5.21 of \cite{NWY2008}, together with Proposition~2.1.1 of Yamanoi \cite{Y04}.

\begin{proposition}\label{thm:HC}
Let $g=(f,\phi):\mathbb{C}^p\rightarrow A \times \mathbb{A}^N$ be a holomorphic map with $T_{\phi}(r)\leq S_{f, H}(r)$ and let \(Z\subset X_k=\overline{j_k^{\delta}(g)(\mathbb C^p)}^{\,\Zar}\) be a subvariety of
\(\codim_{X_k}Z\geq2\). Suppose that $\delta f\not\equiv0$ and every prime divisor in $j_k^{\delta}(g)^*Z$ is transverse to $\delta$ at its general points. Then, for every
\(\eps>0\),
\begin{equation}\label{eq:HC}
 N^{[1]}(r,j_k^{\delta}(g)^*Z)
 \leq_{\mathrm{exc}}\eps T_f(r,H).
\end{equation}
\end{proposition}
\begin{proof} We write $T_f(r)=T_f(r, H)$, and $S_f(r)=S_{f, H}(r)$.

Following the strategy in the proof of Theorem~5.1 of \cite{NWY2008}, we divide the proof into three steps.

\noindent{\bf $\bullet$ Reduction.} 

We first reduce the proof to the case where $A$ admits a splitting $A=B\times C$, with $B$ and $C$ semi-abelian subvarieties, such that
 \begin{align}
\label{5.2}
B &\subset \cap_{l \geq 0} \mathrm{St}_A^0(X_l)
\end{align}
and the composition of $f$ and the second projection $q^B: A\rightarrow A/B=C$ satisfies 
\begin{align}\label{5.3}
T_{q^B\circ f}(r) &=S_f(r).
 \end{align}

This splitting becomes essential a little later, where we explicitly use
$A\cong B\times C$ to deduce $X_{ l }\cong B\times (X_{ l }/B).$

 Let $\mathcal{B}$ be the set of all semi-abelian subvarieties
$B\subset A$ such that
\[
T_{q^B\circ f}(r)=S_f(r).
\]
By Lemma~\ref{cor:commonq}, we have
$\bigcap_{l\geq 0}\St_A^0(X_l)\in\mathcal{B}$, and hence
$\mathcal{B}\neq\emptyset$. Let $B\in\mathcal{B}$ be a minimal element,
that is, if $B'\subset B$ and $B'\in\mathcal{B}$, then $B'=B$.
If $B_1,B_2\in\mathcal{B}$, the finite morphism
$A/(B_1\cap B_2)^0\longrightarrow A/B_1\times A/B_2$
implies that $(B_1\cap B_2)^0\in\mathcal{B}$. Applying this with
$B_1=B$ and
$B_2=\bigcap_{l\geq 0}\St_A^0(X_l)$, the minimality of $B$ gives
\[
B\subset \bigcap_{l\geq 0}\St_A^0(X_l).
\]

Our strategy is, by changing  the pair $A, f$ to $\tilde{A}, \tilde{f}$,  to replace the possibly non-split extension
$$0\rightarrow B\rightarrow A\rightarrow C=A/B\rightarrow 0$$
by a split one, without changing the Nevanlinna problem in any essential way. The construction works as follows.
By the proof of Lemma 5.7 in \cite{NWY2008} (or Lemma 6.5.25 in \cite{NWbook}) we may take a holomorphic map
$u:\mathbb C^p \to A$ such that $q^B \circ u=q^B \circ f$ and 
\begin{equation}
\label{eqn:3081}
T_{u}(r) = S_{f}(r).
\end{equation}

We may assume that the Zariski closure of $u(\mathbb{C}^p)$ is a
semi-abelian subvariety $C'\subset A$ (see \cite{WangCp}). Define the
semi-abelian variety $\tilde A$ by the following base change.
\begin{equation*}
\begin{CD}
\tilde{A}@>p_2>> A \\
@Vp_1VV   @VV{q^B}V \\
C' @>{q^B |_{C'}}>>  C  
\end{CD}
\end{equation*}
Then $\tilde A=\{(c,a)\in C'\times A:q^B(c)=q^B(a)\}$.
The inclusion map $i:C'\to A$ induces a map
$\tau:C'\to\tilde A$ defined by $\tau(x)=(x,i(x))$.
Note that this morphism $\tau$ 
is a section for $p_1:\tilde{A}\to C'$.
Hence this bundle is trivial, i.e. 
$\tilde{A}\cong B\times C'$ and $\tilde{A}/B=C'$.

Put $\tilde{f}=(u,  f):\mathbb C^p \to \tilde{A}$. It is easy to see $\tilde{f}$ has Zariski dense image in $\tilde{A}$ and $\delta\tilde{f}\not\equiv0$.
By \eqref{eqn:3081} we have
\begin{align}
\label{eqn:3093}
T_{f}(r)&\asymp T_{\tilde{f}} (r),\\
\label{eqn:3091}
T_{p_1\circ \tilde{f}} (r) &=  S_{\tilde{f}}(r) .
\end{align}
Set $\tilde{g}=(\tilde{f},\phi):\mathbb{C}^p \rightarrow \tilde{A}\times\mathbb{A}^N$ and denote  $\tilde{X_l}=\overline{j_l^{\delta}(\tilde{g})(\mathbb C^p)}^{\,\Zar}$.

There is one subtlety: after replacing $A, f$ by $\tilde{A}, \tilde{f}$, we must ensure that the same $B$ still lies in the stabilizers of all the new jet closures $\tilde{X_l}$.
To do so, set
\begin{equation}\label{eqn:3095}
B'=\left(B\cap(\cap_{ l \geq0}
\St_{\tilde A}^0(\tilde X_ l ))
\right)^0
\end{equation}
and $p_1':\tilde{A}\to \tilde{A}/B'$ to be the quotient map.
By \eqref{eqn:3091}, we have
\begin{equation}
\label{eqn:3094}
T_{p_1' \circ \tilde{f}}(r) = S_{\tilde{f}}(r).
\end{equation}
Let $q^{B'}:A\to A/B'$ be the quotient map.
Then we have
\begin{equation}\label{eqn:3092}
T_{q^{B '} \circ f}(r) = O(T_{p_1' \circ \tilde{f}}(r)).
\end{equation}
Hence by \eqref{eqn:3093}, \eqref{eqn:3094} and \eqref{eqn:3092},
we conclude $B'\in \mathcal{B}$.
Since $B$ is minimal in $\mathcal{B}$, we get $B'=B$.
By \eqref{eqn:3095}, we have $B \subset \cap_{l\geq 0}
\St_{\tilde{A}}^0(\tilde{X_l})$.

Let $p_{2,k}:\tilde X_k\to X_k$ be the morphism induced by
$p_2:\tilde A\to A$, and set
\[
\tilde Z=p_{2,k}^{-1}(Z)\subset\tilde X_k.
\]
Then
\[
N^{[1]}\bigl(r,j_k^\delta(g)^*Z\bigr)
=
N^{[1]}\bigl(r,j_k^\delta(\tilde g)^*\tilde Z\bigr).
\]

Let $\pi_k: X_k \to X_k/B$ be the quotient map. We also can reduce the proof to the case that the image $\pi_k(Z)$ is  Zariski dense in $X_k/B$. 
Indeed,  assume that the image $\pi_k(Z)$ is not Zariski dense in $X_k/B$,
there is an effective Cartier divisor $E$ on
$X_k/B$ containing $\pi_k(Z)$.
Then we get
\begin{align}
N^{[1]}(r, j_k^{\delta}(g)^*Z) &\leq N(r, (\pi_k\circ j_k^{\delta}(g))^* E)=
O(T_{\pi_k\circ j_k^{\delta}(g)} (r))\leq S_f(r).
\nonumber
\end{align}
Therefore, the proof of Proposition \ref{thm:HC} is finished in this case.   Hence we assume that $\pi_k(Z)$ is Zariski dense in $X_k/B$, and hence, has relative dimension at most $\dim B -2$.

We have that
$$
B \subset \left(\cap_{l \geq 0} \St_A^0(X_l)\right) \cap
\left(\cap_{l \geq 0} \St_{\tilde A}^0(\tilde{X_l})\right) .
$$
Thus $p_{2,k}:\tilde{X_k}  \to X_k$ is $B$-equivariant and
induces a morphism
$p^B_{2,k}: \tilde{X_k}/B \to {X_k}/B.$
Therefore every irreducible component of $\tilde{Z}$ which dominates $\tilde{X_k}/B$ has general fiber dimension at most $\dim B -2$. For an irreducible component $W$ with nondense image in $\tilde{X_k}/B$, the preceding argument gives $N^{[1]}(r, j_k^{\delta}(\tilde{g})^*W) = S_f(r)$.

Hence, by replacing
 $A$ by $\tilde{A}$, $C$ by $C'$, $f$ by $\tilde{f}$ and
$Z$ by an irreducible component of $p_{2,k}^{-1}(Z)$ which dominates $\tilde{X_k}/B$, we achieved our reduction.

 To summarize, we can then assume the following in the sequel: 

\noindent
(i)  There is a semi-abelian subvariety $B \subset A$ satisfying
\begin{align}
\label{5.2}
B &\subset \cap_{l \geq 0} \mathrm{St}_A^0(X_l),\\
\label{5.3}
T_{q^B\circ f}(r) &=S_f(r),\\
\label{5.4}
A &\cong B \times (A/B),
\end{align}
where $q^B:A \to A/B$ is the quotient map.

\noindent
(ii)  $\pi_k(Z)$ is Zariski dense in $X_k/B$.

\noindent{\bf $\bullet$ The Auxiliary divisor $F_l$ and the main sub-lemma.} 

To continue, we choose an auxiliary divisor. Let the notation and assumptions be
as above, and set $C=A/B$. We have
\begin{equation}
\label{5.5}
A\cong B\times C.
\end{equation}
This induces
\[
X_l\cong B\times (X_l/B),\qquad l\geq 0.
\]
Let $\bar B$ be an equivariant compactification of $B$, and put
$\hat X_l=\bar B\times (X_l/B)$. Denote by
\[
\hat\gamma_l:\hat X_l\to\bar B,\qquad
\hat\pi_l:\hat X_l\to X_l/B
\]
the natural projections. We denote by $Z^{\mathrm{ns}}$ the nonsingular locus
of $Z$.

\begin{sublemmam}
\label{lem:jyu}
Let $L \to \bar B$ be an ample line bundle.
Then there is a sequence of natural numbers
$n(1),n(2),n(3),\ldots$ satisfying the following:
\begin{enumerate}
\item
$\lim_{l \to \infty} \frac{n(l)}{l}= 0$.
\item
There exist effective Cartier divisors
$F_l\subset \hat{X}_{k+l}$ and line bundles
$M_l$ on $X_{k+l}/B$ such that
$F_l$ is defined by a non-zero element of
$$
H^0(\hat{X}_{k+l},\hat\gamma _{k+l}^*L^{\otimes n(l)}
\otimes (\hat\pi _{k+l})^{*}M_l)
$$
and that for every point $x\in \mathbb C^p$ with
$j_k^{\delta}(g)(x)\in Z^{\mathrm{ns}}$
$$
\ord_{t=0}{F_l(j_{k+l}^{\delta}(g)(x+ta))}\geq l+1.
$$
\end{enumerate}
\end{sublemmam}

To prove the above Main Sub-lemma, we need the following sub-lemma.
\begin{sublemma}\label{incidence}
For $k\geq0$, we have
$J_{k}(A\times\mathbb{A}^N)\simeq B\times\bigl(J_{k}(A\times\mathbb{A}^N)/B\bigr)$.
Let $\beta_{k}:J_{k}(A\times\mathbb{A}^N)\rightarrow B$ be the first projection, and $\varpi_{k}:J_{k}(A\times\mathbb{A}^N)\rightarrow J_{k}(A\times\mathbb{A}^N)/B$ be the second projection. Let $\pi_{m,k}:J_{m}(A\times\mathbb{A}^N)/B\rightarrow J_{k}(A\times\mathbb{A}^N)/B$ be the canonical morphism for $m\geq k$, and let
\begin{equation}\label{eq:support}
T_{k,l}=
\left\{(\eta,w)\in J_{k+l}(A\times\mathbb{A}^N)\times J_{k}(A\times\mathbb{A}^N)/B~|~
\beta_{k+l}(\eta)=0_B,
~
w=\pi_{k+l,k}\bigl(\varpi_{k+l}(\eta)\bigr)
\right\}.
\end{equation}
Let
$$
\lambda_{l}:J_{k+l}(A\times\mathbb{A}^N)\times J_{k}(A\times\mathbb{A}^N)/B
\longrightarrow J_{k+l}(A\times\mathbb{A}^N)/B,
~
(\eta,w)\longmapsto\varpi_{k+l}(\eta).
$$
There is a closed subscheme
$\mathcal T_{k,l}^{\dagger}
\subset
J_{k+l}(A\times\mathbb{A}^N)\times J_{k}(A\times\mathbb{A}^N)/B$
such that:
\begin{enumerate}
\item $\operatorname{Supp}\mathcal T_{k,l}^{\dagger}=T_{k,l}$.
\item The restriction
$\lambda_{l}'=\lambda_{l}|_{\mathcal T_{k,l}^{\dagger}}$ is finite,
is a homeomorphism on underlying topological spaces, and
$(\lambda_{l}')_*\mathcal O_{\mathcal T_{k,l}^{\dagger}}$ is locally free of
rank $l+1$ over 
$\bigl(J_{k+l}(A\times\mathbb{A}^N)/B\bigr)^{\mathrm{reg}}=C\times\mathbb{A}^N\times J_{k+l,A\times\mathbb{A}^N}^{\mathrm{reg}}$, where $J_{k+l,A\times\mathbb{A}^N}^{\mathrm{reg}}$ is the open locus of regular jets.
\item If $h:(\Delta,0)\to A\times\mathbb{A}^N$ is a holomorphic germ satisfying
$h_B(0)=0_B$, let
$
w_0=\varpi_k(j_kh(0)),
\widetilde{j_{k+l}h}(t)=(j_{k+l}h(t),w_0).
$
Then
\begin{equation}\label{eq:contact}
\operatorname{ord}_0
\bigl(\widetilde{j_{k+l}h}\bigr)^*
(\mathcal T_{k,l}^{\dagger})\geq l+1.
\end{equation}
\end{enumerate}
\end{sublemma}

\begin{proof}
Let $0_{A\times\mathbb{A}^N}=(0_B,0_{C\times\mathbb{A}^N})\in A\times\mathbb{A}^N$, and put
$J_{m,A\times\mathbb{A}^N}=J_m(A\times\mathbb{A}^N)_{0_{A\times\mathbb{A}^N}}.$
We have the trivializations
$$
J_m(A\times\mathbb{A}^N)\simeq A\times\mathbb{A}^N\times J_{m,A\times\mathbb{A}^N},
\qquad
J_m(A\times\mathbb{A}^N)/B\simeq C\times\mathbb{A}^N\times J_{m,A\times\mathbb{A}^N}.
$$
By applying Proposition~2.1.1 in \cite{Y04} (or Lemma 6.5.40 in \cite{NWbook}) to $A\times\mathbb{A}^N$ at $0_{A\times\mathbb{A}^N}$, we get the following:  Let
$\mathcal T_{k,l}
\subset
J_{k+l}(A\times\mathbb{A}^N)\times J_{k,A\times\mathbb{A}^N}$
be the nonreduced closed subscheme. Its support is
\[
\operatorname{Supp}\mathcal T_{k,l}
=
\left\{
\bigl(0_{A\times\mathbb{A}^N},\eta,\eta_{k+l, k}\bigr)~|~
\eta\in J_{k+l,A\times\mathbb A^N}
\right\},
\]
where $\eta_{k+l, k}\in J_{k,A\times\mathbb A^N}$ denotes the truncation of
$\eta$ to $k$-jet coordinate.

Let
$\lambda_{k,l}:\mathcal T_{k,l}\longrightarrow J_{k+l,A\times\mathbb{A}^N}$
be its jet projection. Then:
\begin{itemize}
\item $\lambda_{k,l}$ is finite and is a homeomorphism on the underlying
topological spaces of $\mathcal T_{k,l}$ and $J_{k+l,A\times\mathbb{A}^N}$.
\item
$(\lambda_{k,l})_*\mathcal O_{\mathcal T_{k,l}}$
is locally free of rank $l+1$ over
$J_{k+l,A\times\mathbb{A}^N}^{\mathrm{reg}}$.
\item For a holomorphic germ $f:(\Delta,0)\to A\times\mathbb{A}^N$ satisfying $f(0)=0_{A\times\mathbb{A}^N}$, the closed subscheme
$\mathcal T_{j_k(f)(0)}\subset J_{k+l}(A\times\mathbb{A}^N),$
which is the fiber of $\mathcal T_{k,l}$ over
$j_k(f)(0)\in J_{k,A\times\mathbb{A}^N},$
satisfies
\begin{equation}\label{eq:contact}
\operatorname{mult}_0
j_{k+l}(f)\cdot \mathcal T_{j_k(f)(0)}
\geq l+1.
\end{equation}

\end{itemize}

Define
$$
\Theta_l:
C\times\mathbb{A}^N\times J_{k+l}(A\times\mathbb{A}^N)\times J_{k,A\times\mathbb{A}^N}
\longrightarrow
J_{k+l}(A\times\mathbb{A}^N)\times J_k(A\times\mathbb{A}^N)/B
$$
by
$$
\Theta_l\bigl(c,a,(b,c',a',v),v'\bigr)
=
\bigl((b,c'+c,a+a',v),(c,a,v')\bigr),
$$
where $c,c'\in C$, $a,a'\in\mathbb{A}^N$, $v\in J_{k+l,A\times\mathbb{A}^N}$, $b\in B$, and $v'\in J_{k,A\times\mathbb{A}^N}$. It is easy to see that $\Theta_l$ is an isomorphism. Indeed, its inverse is given by
$$
\Theta_l^{-1}
\bigl((b,\widetilde c,\widetilde a,v),(c,a,v')\bigr)=\bigl(c,a,(b,\widetilde c-c,\widetilde a-a,v),v'\bigr), 
$$where $b\in B$, $c,\widetilde c\in C$, $a,\widetilde a\in\mathbb A^N$, $v\in J_{k+l,A\times\mathbb A^N}$, and $v'\in J_{k,A\times\mathbb A^N}$.
Set
$\mathcal T_{k,l}^{\dagger}=\Theta_l\bigl(C\times\mathbb{A}^N\times\mathcal T_{k,l}\bigr).$
Since $\Theta_l$ is an isomorphism, $\mathcal T_{k,l}^{\dagger}$ is a closed subscheme of
$J_{k+l}(A\times\mathbb{A}^N)\times J_k(A\times\mathbb{A}^N)/B$.
Moreover, the definition of $\Theta_l$ gives
$
\operatorname{Supp}\mathcal T_{k,l}^{\dagger}=T_{k,l}.
$
This is (i).

Let
$s:J_{k+l}(A\times\mathbb{A}^N)\times J_{k,A\times\mathbb{A}^N}\longrightarrow C\times\mathbb{A}^N$
be the projection induced by
$J_{k+l}(A\times\mathbb{A}^N)\simeq B\times C\times\mathbb{A}^N\times J_{k+l,A\times\mathbb{A}^N}.$
Denote
$\chi:C\times\mathbb{A}^N\times\mathcal T_{k,l}
\longrightarrow
C\times\mathbb{A}^N\times\mathcal T_{k,l}$
by
$\chi(c,a,\tau)=((c,a)+s(\tau),\tau)$, where $(c,a)\in C\times\mathbb{A}^N, \tau\in \mathcal T_{k,l}$.
This is an automorphism. Then we have $
\lambda_l' \circ
\left.\Theta_l\right|_{(C\times\mathbb{A}^N)\times\mathcal{T}_{k,l}}
=
\left(\operatorname{id}_{C\times\mathbb{A}^N}\times\lambda_{k,l}\right)
\circ\chi$.
Consequently, $\lambda_l'$ is isomorphic to
$\operatorname{id}_{C\times\mathbb{A}^N}\times\lambda_{k,l}$, which proves (ii).

\[
\begin{CD}
(C\times\mathbb{A}^N)\times\mathcal{T}_{k,l}
@>{\chi}>>
(C\times\mathbb{A}^N)\times\mathcal{T}_{k,l}
\\
@V{\left.\Theta_l\right|_{(C\times\mathbb{A}^N)
\times\mathcal{T}_{k,l}}}V{\simeq}V
@VV{\operatorname{id}_{C\times\mathbb{A}^N}\times\lambda_{k,l}}V
\\
\mathcal{T}_{k,l}^{\dagger}
@>{\lambda_l'}>>
(C\times\mathbb{A}^N)\times J_{k+l,A\times\mathbb{A}^N}.
\end{CD}
\]

Finally, let $h$ be as in (iii), and put
$\tilde{h}=h-h_{C\times\mathbb{A}^N}(0).$
Since $h_B(0)=0_B$, we have $\tilde{h}(0)=0_{A\times\mathbb{A}^N}$. Under
$J_k(A\times\mathbb{A}^N)/B
\simeq
C\times\mathbb{A}^N\times J_{k,A\times\mathbb{A}^N},$
one has
$w_0=\bigl(h_{C\times\mathbb{A}^N}(0),j_k\tilde{h}(0)\bigr).$
Moreover,
$\Theta_l\bigl(h_{C\times\mathbb{A}^N}(0),j_{k+l}\tilde{h}(t),j_k\tilde{h}(0)\bigr)=\bigl(j_{k+l}h(t),w_0\bigr).$
Since $\Theta_l$ is an isomorphism, applying \eqref{eq:contact} to $\tilde{h}$ gives
$$
\operatorname{ord}_0
\bigl(\widetilde{j_{k+l}h}\bigr)^*
(\mathcal T_{k,l}^{\dagger})
\geq l+1.
$$
This proves (iii).
\end{proof}

Now we prove the Main Sub-Lemma \ref{lem:jyu}.

\begin{proof}

Let \(\bar Z\) be the Zariski closure of \(Z\) in \(\widehat X_k\), and let \(r_1:Z^{\dagger}\to \bar Z\) be a desingularization of \(\bar Z\) which is an isomorphism over \(Z^{\mathrm{ns}}\). Put \(Y_k=X_k/B\). Consider the sequence of morphisms
\begin{equation}
\label{eqn:sqm}
Z^{\mathrm{ns}}\overset{r_0}{\hookrightarrow} Z^{\dagger}\overset{r_1}{\longrightarrow}\bar Z\overset{r_2}{\hookrightarrow}\widehat X_k\overset{\widehat\pi_k}{\longrightarrow}Y_k,
\end{equation}
where \(r_0\) and \(r_1\circ r_0\) are open immersions and \(r_2\) is a closed immersion. Put \(r=\widehat\pi_k\circ r_2\circ r_1:Z^\dagger\to Y_k\). By generic flatness, there exists a nonempty nonsingular Zariski open subset \(Y_k^{\mathrm{fl}}\subset Y_k\) such that \(r:r^{-1}(Y_k^{\mathrm{fl}})\to Y_k^{\mathrm{fl}}\) is flat. In particular, every fiber of \(r\) over \(Y_k^{\mathrm{fl}}\) has dimension \(\dim Z^\dagger-\dim Y_k\).

Consider the pullback of \eqref{eqn:sqm} by the natural projection \(B\times Y_k\to Y_k\):
$$
B\times Z^{\mathrm{ns}}\overset{s_0}{\hookrightarrow}B\times Z^{\dagger}\overset{s_1}{\longrightarrow}B\times\bar Z\overset{s_2}{\hookrightarrow}B\times\widehat X_k\overset{s_3}{\longrightarrow}B\times Y_k.
$$
Put \(s=s_3\circ s_2\circ s_1:B\times Z^{\dagger}\to B\times Y_k\). Then \(s(a,z)=(a,r(z))\). Let \(L\) be an ample line bundle on \(\bar B\) and define
\begin{equation}
\label{phi}
\mu:B\times\widehat X_k\longrightarrow\bar B,\qquad \mu(a,w)=a+\widehat\gamma_k(w).
\end{equation}
Let \(L_1^{\dagger}\) be the pullback of \(L\) by the composition \(B\times Z^{\dagger}\overset{s_2\circ s_1}{\longrightarrow}B\times\widehat X_k\overset{\mu}{\longrightarrow}\bar B\).

Since the restriction of \(s\) over \(B\times Y_k^{\mathrm{fl}}\) is flat, the upper semicontinuity theorem \cite[p.~288]{AG} implies that, for every \(n>0\), there is a nonempty Zariski open subset \(U_n\subset B\times Y_k^{\mathrm{fl}}\) such that $\dim_{\mathbb C}H^0\bigl((B\times Z^{\dagger})|_P,L_{1,P}^{\dagger\otimes n}\bigr)$
is constant for \(P\in U_n\). Put this constant by \(G_n\). Here \((B\times Z^{\dagger})|_P\) denotes the fiber of \(s:B\times Z^{\dagger}\to B\times Y_k\) over \(P\), and \(L_{1,P}^{\dagger\otimes n}\) denotes the induced line bundle. Since \(\bigcap_{n\geq1}U_n\neq\varnothing\), choose \((a,y)\in\bigcap_{n\geq1}U_n\). Replacing \(L\) by its pullback under the translation \(B\ni x\mapsto x+a\in B\), we may assume that \(a=0_B\).

Now for a positive integer $l>0$, let $\mathcal T_{k,l}^{\dagger} \subset J_{k+l}(A\times\mathbb{A}^N)\times J_{k}(A\times\mathbb{A}^N)/B$ be the closed subscheme, and let $
\lambda_{l}:J_{k+l}(A\times\mathbb{A}^N)\times J_{k}(A\times\mathbb{A}^N)/B\longrightarrow J_{k+l}(A\times\mathbb{A}^N)/B,
$ be the morphism in Sub-Lemma $\ref{incidence}$. Then $\lambda_l$ has the following properties;
\begin{enumerate}
\item $\operatorname{Supp}\mathcal T_{k,l}^{\dagger}=T_{k,l}$.
\item The restriction
$\lambda_{l}'=\lambda_{l}|_{\mathcal T_{k,l}^{\dagger}}$ is finite,
is a homeomorphism on underlying topological spaces, and
$(\lambda_{l}')_*\mathcal O_{\mathcal T_{k,l}^{\dagger}}$ is locally free of
rank $l+1$ over $\bigl(J_{k+l}(A\times\mathbb{A}^N)/B\bigr)^{\mathrm{reg}}$.
\item If $h:(\Delta,0)\to A\times\mathbb{A}^N$ is a holomorphic germ satisfying
$h_B(0)=0_B$, and
$
w_0=\varpi_k(j_kh(0)),
\widetilde{j_{k+l}h}(t)=(j_{k+l}h(t),w_0),
$
then
\begin{equation}\label{eq:contact2}
\operatorname{ord}_0
\bigl(\widetilde{j_{k+l}h}\bigr)^*
(\mathcal T_{k,l}^{\dagger})\geq l+1.
\end{equation}
\end{enumerate}

Since $Y_{k+l}$ is a Zariski closed subset of $C\times\mathbb{A}^N\times J_{k+l,A\times\mathbb{A}^N}$, we have the closed immersion 
\begin{equation}
\label{eqn:503}
\begin{aligned}
B\times Y_{k+l}\times Y_k
&\subset
B\times C\times\mathbb{A}^N
\times J_{k+l,A\times\mathbb{A}^N}
\times C\times\mathbb{A}^N
\times J_{k,A\times\mathbb{A}^N}
\\
&=
J_{k+l}(A\times\mathbb{A}^N)
\times
\bigl(J_k(A\times\mathbb{A}^N)/B\bigr).
\end{aligned}
\end{equation}

Let $\mathcal{S}_l\subset B\times Y_{k+l}\times Y_{k}$
be the closed subscheme obtained by the pullback of
$\mathcal{T}_{k,l}^{\dagger}$ by \eqref{eqn:503}.
Let $q:\mathcal{S}_l\to Y_{k+l}$ be the composition with
the second projection
$B\times Y_{k+l}\times Y_{k}\to Y_{k+l}$.
We put
$$
Y_{k+l}^{\mathrm{reg}}=Y_{k+l}\cap (C\times\mathbb{A}^N\times J_{k+l,A\times\mathbb{A}^N}^{\mathrm{reg}}),
$$
 which is the Zariski open subset of $Y_{k+l}$. Since $\delta f\not\equiv0$, we obtain $Y_{k+l}^{\mathrm{reg}}\not=\emptyset $. 
Then by the above properties of $\lambda$,
we have the corresponding properties for $q$;
\begin{enumerate}
\item
$q$ is finite,
\item
the direct image sheaf $q _*\mathcal{O}_{\mathcal{S} _l}$
is locally generated by $l+1$ elements as 
$\mathcal{O}_{Y_{k+l}}$-module on $Y_{k+l}^{\mathrm{reg}}$,
\item $q$ gives the isomorphism of underlying topological spaces of
$\mathcal{S}_l$ and $Y_{k+l}$.
\end{enumerate}

We consider the following commutative diagram \eqref{eqn:cdm}
obtained by the base change of \eqref{eqn:sqm} with a sequence
of morphisms
$$
\mathcal{S}_l \hookrightarrow  B\times Y_{k+l}\times Y_{k}\to
B\times Y_{k}\to Y_{k}.
$$
Here $B\times Y_{k+l}\times Y_k\longrightarrow B\times Y_k$ is the natural projection.
\begin{equation}
\label{eqn:cdm}
\begin{CD}
\mathcal{Z}_l^\mathrm{ns} @>>> B\times Y_{k+l}\times Z^\mathrm{ns}@>>>
B\times Z^\mathrm{ns}@>>> Z^\mathrm{ns}\\
@VVu_0V @VVt_0V @VVs_0V @VVr_0V \\
\mathcal{Z}_l^{\dagger} @>>> B\times Y_{k+l}\times Z^{\dagger}@>>>
B\times Z^{\dagger}@>>> Z^{\dagger}\\
@VVu_1V @VVt_1V @VVs_1V @VVr_1V \\
\mathcal{Z}_l @>v'>> B\times Y_{k+l}\times \bar{Z}@>>>
B\times \bar{Z}@>>> \bar{Z}\\
@VVu_2V @VVt_2V @VVs_2V @VVr_2V \\
\cdot @>>> B\times Y_{k+l}\times \hat{X}_{k}@>>>
B\times \hat{X}_{k}@>>> \hat{X}_{k}\\
@VVu_3V @VVt_3V @VVs_3V @VV \hat\pi_k V \\
\mathcal{S} _l @>v>> B\times Y_{k+l}\times Y_{k}@>>>
B\times Y_{k}@>>> Y_{k}
\end{CD}
\end{equation}
Let \(\mathcal{L}_l^{\dagger}\) be the line bundle on \(\mathcal{Z}_l^{\dagger}\) obtained by the pullback of \(L_1^{\dagger}\) by the morphisms in \eqref{eqn:cdm}, and let \(\mathcal{S}_{l,n}\) be the nonempty Zariski open subset of \(\mathcal{S}_l\) obtained as the inverse image of \(U_n\). Since \(\dim H^0((B\times Z^{\dagger})|_P,L_{1,P}^{\dagger\otimes n})=G_n\) for \(P\in U_n\), Grauert's theorem \cite[p.~288]{AG} implies that \(s_*L_1^{\dagger\otimes n}\) is locally free of rank \(G_n\) on \(U_n\) and that the natural map
$$
s_*L_1^{\dagger\otimes n}\otimes\mathbb C(P)\longrightarrow H^0((B\times Z^{\dagger})|_P,L_{1,P}^{\dagger\otimes n})
$$
is an isomorphism for \(P\in U_n\).

Put \(u=u_3\circ u_2\circ u_1:\mathcal{Z}_l^{\dagger}\to\mathcal{S}_l\), where \(u_1,u_2,u_3\) are the morphisms in \eqref{eqn:cdm}. By the theorem of cohomology and base change \cite[p.~290]{AG}, the natural map
$$
u_*\mathcal{L}_l^{\dagger\otimes n}\otimes\mathbb C(P)\longrightarrow H^0(\mathcal{Z}_l^{\dagger}|_P,\mathcal{L}_{l,P}^{\dagger\otimes n})
$$
is an isomorphism for \(P\in\mathcal{S}_{l,n}\). It follows that $u_*\mathcal{L}_l^{\dagger\otimes n}$ is locally generated by $G_n$ elements as an $\mathcal{O}_{\mathcal{S}_l}$-module on $\mathcal{S}_{l,n}\subset\mathcal{S}_l$. Put \(Y_{k+l,n}=q(\mathcal{S}_{l,n})\). Since the underlying topological spaces of \(\mathcal{S}_l\) and \(Y_{k+l}\) are the same, \(Y_{k+l,n}\) is a nonempty Zariski open subset of \(Y_{k+l}\). By the above properties of \(q\), $(q\circ u)_*\mathcal{L}_l^{\dagger\otimes n}$ is locally generated by $(l+1)G_n$ elements as an $\mathcal{O}_{Y_{k+l}}$-module on $Y_{k+l,n}\cap Y_{k+l}^{\mathrm{reg}}$. Here, since \(Y_{k+l}^{\mathrm{reg}}\neq\emptyset\) and \(Y_{k+l}\) is irreducible, \(Y_{k+l,n}\cap Y_{k+l}^{\mathrm{reg}}\neq\emptyset\).

Now look at the following commutative diagram
\begin{equation*}
\begin{CD}
\mathcal{Z}_l^\mathrm{ns}\\
@VVu_0V \\
\mathcal{Z}_l ^{\dagger}@>t_2\circ v'\circ u_1>>
B\times Y_{k+l}\times \hat{X}_{k} @>\psi >>
\bar{B}\times Y_{k+l}@>\rho >> \bar{B}\\
@VVq\circ uV @VV \text{2nd proj} V @VV\tau V \\
Y_{k+l} @= Y_{k+l} @=  Y_{k+l}
\end{CD}
\end{equation*}
where \(\rho\) is the first projection, \(\tau\) is the second projection, and
$$
\psi:B\times Y_{k+l}\times\hat X_k\longrightarrow \bar B\times Y_{k+l}
$$
is the morphism defined by \(\psi(a,v,w)=\bigl(a+\hat\gamma_k(w),v\bigr)\).
Since
$(\rho \circ \psi \circ t_2\circ v'\circ u_1)^*L=\mathcal{L}_l^{\dagger}$,
we have a natural morphism
\begin{equation}
\label{eqn:530}
\tau _{*}\rho ^{*}L^{\otimes n}=
H^{0}(\bar{B},L^{\otimes n})\otimes _{\mathbb{C}  }
\mathcal{O}_{Y_{k+l}} \to (q\circ u)_{*}\mathcal{L}_l^{\dagger \otimes n}.
\end{equation}
Here, note that $\rho \circ \psi =\mu \circ \beta$ where
$\beta :B\times Y_{k+l}\times \hat{X}_{k}\to B\times \hat{X}_{k}$
is the morphism in the diagram \eqref{eqn:cdm} and
$\mu$ was defined by (\ref{phi}).

Now put \(I_n=\dim_{\mathbb C}H^0(\bar B,L^{\otimes n})\). Then there exist a positive integer \(n_0\) and positive constants \(C_1,C_2\) such that
$$
I_n>C_1n^{\dim\bar B},\qquad G_n<C_2n^{\dim\bar B-2}
$$
for \(n>n_0\). Note that \(G_n=\dim_{\mathbb C}H^0((B\times Z^\dagger)|_P,L_{1,P}^{\dagger\otimes n})\) for \(P\in\bigcap_{n\geq1}U_n\). By \(\operatorname{codim}_{\hat X_k}\bar Z\geq2\) and the dominance of \(\hat\pi_k\circ r_2:\bar Z\to Y_k\), the fiber \((B\times Z^\dagger)|_P=s^{-1}(P)\) has dimension at most \(\dim\bar B-2\). Hence the above estimate for \(G_n\) follows.
 Hence, for a positive integer \(l\), we can choose a positive integer \(n(l)\), for example \(n(l)\sim l^{3/4}\), such that
$$
I_{n(l)}>(l+1)G_{n(l)},\qquad \lim_{l\to\infty}\frac{n(l)}{l}=0.
$$
Let \(\mathcal F\) be the kernel of \eqref{eqn:530} for \(n=n(l)\). Then we have an exact sequence
$$
0\to\mathcal F\to\tau_*\rho^*L^{\otimes n(l)}\to(q\circ u)_*\mathcal L_l^{\dagger\otimes n(l)}.
$$
It follows from \(I_{n(l)}>(l+1)G_{n(l)}\) that \(\mathcal F\neq0\). By tensoring \(\mathcal F\) with a sufficiently ample line bundle \(M_l\) on \(Y_{k+l}\), we may assume that \(H^0(Y_{k+l},\mathcal F\otimes M_l)\neq0\).

Since we have
\begin{equation*}
\begin{split}
H^0(Y_{k+l},\mathcal{F} \otimes M_l)&\subset
H^0(Y_{k+l},(\tau _{*}\rho ^{*}L^{\otimes n(l)})\otimes M_l)\\
&=H^0(Y_{k+l},\tau _{*}(\rho ^{*}L^{\otimes n(l)}\otimes \tau ^*M_l))\\
&=H^0(\bar{B}\times Y_{k+l},\rho ^{*}L^{\otimes n(l)}\otimes \tau ^*M_l),
\end{split}
\end{equation*}
we may take a divisor \(F_l\subset \bar B\times Y_{k+l}\) defined by a non-zero global section of \(H^0(Y_{k+l},\mathcal F\otimes M_l)\). Then
$$
\mathcal Z_l^{\mathrm{ns}}\subset \psi^*F_l.
$$
Here \(\mathcal Z_l^{\mathrm{ns}}\subset\mathcal Z_l\) is an open subscheme, and
$\mathcal Z_l\overset{t_2\circ v'}{\hookrightarrow}B\times Y_{k+l}\times\hat X_k$ is a closed subscheme.

Using the decomposition \(A=B\times C\), let \(f_B:\mathbb C^p\to B\) and \(f_C:\mathbb C^p\to C\) be the compositions of \(f\) with the first and second projections of \(A=B\times C\), respectively.

Under the identification \(X_m=B\times Y_m\), we have
$j_m^\delta(g)=(f_B,\hat\pi_m\circ j_m^\delta(g)).$
Fix \(x\in\mathbb C^p\) such that
$j_k^\delta(g)(x)\in Z^{\mathrm{ns}},$
and define the holomorphic germ
\[
\widetilde h_x:(\Delta,0)
\longrightarrow B\times Y_{k+l}\times Z^{\mathrm{ns}}
\]
by
\[
\widetilde h_x(t)
=
\left(
f_B(x+ta)-f_B(x),
\hat\pi_{k+l}\circ j_{k+l}^\delta(g)(x+ta),
j_k^\delta(g)(x)
\right).
\]
Since
$\hat\pi_k\circ j_k^\delta(g)(x)=\pi_{k+l,k}\bigl(\hat\pi_{k+l}\circ j_{k+l}^\delta(g)(x)\bigr),$
we have
$\widetilde h_x(0)\in\operatorname{Supp}\mathcal Z_l^{\mathrm{ns}}.$
Moreover,
$(\psi\circ\widetilde h_x)(t)=j_{k+l}^\delta(g)(x+ta)$.

Let \(\iota_l\) denote the closed immersion in \eqref{eqn:503}, and let
$\vartheta:B\times Y_{k+l}\times Z^{\mathrm{ns}}\longrightarrow B\times Y_{k+l}\times Y_k$
be the morphism induced by the natural projection
\(Z^{\mathrm{ns}}\to Y_k\). We have
\begin{equation}\label{eq:order-base-change}
\operatorname{ord}_{t=0}
\widetilde h_x^{*}(\mathcal Z_l^{\mathrm{ns}})
=
\operatorname{ord}_{t=0}
(\vartheta\circ\widetilde h_x)^{*}(\mathcal S_l).
\end{equation}

Define
$h_x(t)=\left(f_B(x+ta)-f_B(x),f_C(x+ta),\phi(x+ta)\right)\in A\times\mathbb A^N,$ where $\phi$ is the $\mathbb A^N$-component of $g$ and denote
$w_x=\hat\pi_k\bigl(j_kh_x(0)\bigr).$
We have
$$
j_{k+l}h_x(t)
=
\left(
f_B(x+ta)-f_B(x),
\hat\pi_{k+l}\circ j_{k+l}^\delta(g)(x+ta)
\right).
$$
Consequently,
\begin{equation}\label{eq:translated-jet-germ}
\bigl(\iota_l\circ \vartheta\circ\widetilde h_x\bigr)(t)
=
\bigl(j_{k+l}h_x(t),w_x\bigr)
=
\widetilde{j_{k+l}h_x}(t).
\end{equation}

It follows from
\eqref{eq:order-base-change},
\eqref{eq:translated-jet-germ}, and Sub-Lemma~\ref{incidence} that
\begin{align}
\operatorname{ord}_{t=0}
\widetilde h_x^{*}(\mathcal Z_l^{\mathrm{ns}})
=
\operatorname{ord}_{t=0}
(\vartheta\circ\widetilde h_x)^{*}(\mathcal S_l)
\nonumber
=
\operatorname{ord}_{t=0}
\bigl(\iota_l\circ \vartheta\circ\widetilde h_x\bigr)^{*}
\bigl(\mathcal T_{k,l}^{\dagger}\bigr)
\nonumber
=
\operatorname{ord}_{t=0}
\bigl(\widetilde{j_{k+l}h_x}\bigr)^{*}
\bigl(\mathcal T_{k,l}^{\dagger}\bigr)
\nonumber
\ge l+1.
\label{eq:contact-order}
\end{align}

On the other hand, the section defining \(F_l\) vanishes on \(\mathcal Z_l^{\mathrm{ns}}\). Hence
\begin{align}
\operatorname{ord}_{t=0}
F_l\bigl(j_{k+l}^\delta(g)(x+ta)\bigr)
=
\operatorname{ord}_{t=0}
\widetilde h_x^*\psi^*F_l
\nonumber
\ge
\operatorname{ord}_{t=0}
\widetilde h_x^{*}(\mathcal Z_l^{\mathrm{ns}})
\nonumber
\ge l+1.
\end{align}

Here \(F_l\) is regarded as a divisor on
$\widehat X_{k+l}=\overline B\times Y_{k+l},$
and the second projection
$\tau:\overline B\times Y_{k+l}\longrightarrow Y_{k+l}$
corresponds to \(\widehat\pi_{k+l}\) under this identification.

\end{proof}

\begin{sublemma}\label{lem:discDVR}
Let \(E\subset\mathbb C^p\) be an irreducible analytic hypersurface and let
\(\delta=\sum_{j=1}^pa_j\frac{\partial}{\partial z_j}\) be transverse to \(E\) at a general point.  For a fixed
nonzero meromorphic section \(\sigma\), we have
\[
 \ord_{t=0}\sigma(x+ta)=\nu_E(\sigma)
\]
for general \(x\in E_{\reg}\).
\end{sublemma}

\begin{proof}
Write \(\sigma=h^mv\), where \(h=0\) is a reduced equation of \(E\), \(m=\nu_E(\sigma)\), and \(v\) is a meromorphic section satisfying \(\nu_E(v)=0\). For general \(x\in E_{\reg}\), \(\delta h(x)\neq0\) and \(v\) is holomorphic and nonvanishing at \(x\). We have $h(x+ta)=t\,\delta h(x)+O(t^2),~ v(x)\neq0.$ The equality follows.
\end{proof}

\noindent{\bf $\bullet$ The end of the proof.} 

Let $E$ be a prime divisor in $j_k^{\delta}(g)^*(Z^\text{ns})$. 
For fixed \(l\), applying Sub-Lemma
\ref{lem:discDVR} to the pullback defining section
of \(F_l\), we get
\[
 \nu_{E}(j_{k+l}^{\delta}(g)^*F_l)\geq l+1.
\]
Therefore,
\begin{equation}\label{eq:HC1}
 (l+1)N^{[1]}(r,j_{k}^{\delta}(g)^*Z^\text{ns})
 \leq N(r,j_{k+l}^{\delta}(g)^*F_l).
\end{equation}
By the First Main Theorem,
\begin{align}
 N(r,j_{k+l}^{\delta}(g)^*F_l)
 &\leq T_{j_{k+l}^{\delta}(g)}(r,F_l)+O(1)\notag\\
 &\leq n(l)T_{f_B}(r,L)
       +T_{\hat \pi_{k+l}\circ j_{k+l}^{\delta}(g)}(r,M_l)+O(1)\notag\\
 &\leq n(l)T_{f_B}(r,L)+S_{f}(r).\label{eq:HC2}
\end{align}

Given \(\eps>0\), choose \(l\) so large that
$ \frac{n(l)}{l+1}<\frac{\eps}{2}.
$
This proves \eqref{eq:HC} on \(Z^{\text{ns}}\).
Every irreducible component of \(Z\setminus Z^\text{ns}\) has codimension at least two in \(X_k\). By induction, we obtain Proposition \ref{thm:HC}.  

Thus the proof of Proposition \ref{thm:HC} is completed.

\end{proof}

 Proposition \ref{thm:HC} gives the following key result.
\begin{theorem}[The key result: a GCD-type theorem]\label{cor:HC-applications}
\begin{enumerate}[label=(\roman*), leftmargin=1.5em, labelsep=0.4em]

\item Let \(Z\subset D\) be a subvariety such that \(\operatorname{codim}_A Z\geq 2\). Then
for every $\varepsilon>0$,
\[
N^{[1]}(r,f^*Z)
\leq_{\mathrm{exc}}
\varepsilon T_f(r,H).
\]

\item Let \(X=\overline{j_{p,1}f(\mathbb C^p)}^{\,\mathrm{Zar}}\), and let \(Z\subset X\cap J_{p,1}(D)\) satisfy \(\operatorname{codim}_X Z\geq2\). Then, for every
$\varepsilon>0$,
\[
N^{[1]}\!\left(r,(j_{p,1}f)^*Z\right)
\leq_{\mathrm{exc}}
\varepsilon T_f(r,H).
\]
\end{enumerate}
    
\end{theorem}
\begin{proof}  

Write
$\operatorname{Supp}f^*D=\bigcup_{\nu}E_\nu$
as the union of its irreducible components. By
Lemma~\ref{lem:generic}, we may choose
$\delta=\sum_{j=1}^p a_j\frac{\partial}{\partial z_j}$ such that \(\delta\) is transverse to every \(E_\nu\) at a general smooth point of \(E_\nu\) and $\delta f\not\equiv0$. 

We first prove (i). Let
$Z\subset D,~\operatorname{codim}_A Z\ge2.$
We apply Proposition \ref{thm:HC} with
$g=f,~ k=0.$
Note that in this case \(N=0\), so there is no \(\phi\)-part. 
Since
$j_0^\delta(g)=g=f$ and \(f({\Bbb C}^p)\) is Zariski dense in \(A\), we have
$X_0=\overline{f({\Bbb C}^p)}^{\,\rm Zar}=A.$
Thus
$\operatorname{codim}_{X_0}Z=\operatorname{codim}_A Z\ge2.$

It remains only to verify the transversality assumption in
Proposition \ref{thm:HC}.  Since \(Z\subset D\),
$f^{-1}(Z)\subset f^{-1}(D), $
and therefore
$\operatorname{Supp} f^*Z\subset\operatorname{Supp} f^*D.$
Let \(E\) be a prime divisor occurring in \(f^*Z\).  Then
$E\subset\operatorname{Supp}f^*D.$
Since \(E\) is an irreducible hypersurface, it must be one of the
irreducible components \(E_\nu\) of
\(\operatorname{Supp}f^*D\).  By the choice of \(\delta\), $\delta$ is transverse to 
$E$ at a general smooth point of \(E\).

Hence all the assumptions of Proposition \ref{thm:HC} are satisfied.
Consequently,
\[
N^{[1]}(r,f^*Z)
\le_{\rm exc}
\varepsilon T_f(r,H),
\]
which proves~(i).

\bigskip
To prove (ii), we may write
$j_{p,1}f=\bigl(f,\xi_{p,1}(f)\bigr),$
where \(\xi_{p,1}(f)\) denotes the jet coordinate part.
By Lemma~2.2,
\[
T_{\xi_{p,1}(f)}(r)
\le 
S_{f, H}(r).
\]
Thus
$g=j_{p,1}f=\bigl(f,\xi_{p,1}(f)\bigr)$
is an admissible map in Proposition \ref{thm:HC}.

We again apply Proposition \ref{thm:HC} with
$k=0.$
Then
$j_0^\delta(g)=g=j_{p,1}f,$
and hence
$X_0=\overline{g({\Bbb C}^p)}^{\,\rm Zar}=\overline{j_{p,1}f({\Bbb C}^p)}^{\,\rm Zar}=X.$
Therefore
$\operatorname{codim}_{X_0}Z=\operatorname{codim}_X Z\ge2.$

We next verify the transversality condition.  Since
$Z\subset J_{p,1}(D),$
we have
$\operatorname{Supp}(j_{p,1}f)^*Z
\subset
\operatorname{Supp}(j_{p,1}f)^*J_{p,1}(D).$
Moreover,
\begin{equation}\label{eq:support-pjet-divisor}
\operatorname{Supp}
(j_{p,1}f)^*J_{p,1}(D)
\subset
\operatorname{Supp}f^*D.
\end{equation}
Indeed, if \(s=0\) is a reduced local defining equation for \(D\),
then \(J_{p,1}(D)\) is locally defined by
$s=0,~d_{\bar u}s=0,~\bar u\in W_{p,\le1}.$
Therefore
$(j_{p,1}f)^{-1}\bigl(J_{p,1}(D)\bigr)\subset f^{-1}(D), $
which gives~\eqref{eq:support-pjet-divisor}.

Now let \(E\) be a prime divisor occurring in
$(j_{p,1}f)^*Z.$
Then
\[
E\subset
\operatorname{Supp}(j_{p,1}f)^*Z
\subset
\operatorname{Supp}f^*D.
\]
Thus \(E\) is one of the irreducible components
\(E_\nu\) of \(\operatorname{Supp}f^*D\).  By our choice of
\(\delta\), $\delta$ is transverse to $E$ 
at a general smooth point.

Hence all the assumptions of Proposition \ref{thm:HC} are satisfied for
$g=j_{p,1}f,~ k=0.$
It follows that
\[
N^{[1]}\bigl(r,(j_{p,1}f)^*Z\bigr)=N^{[1]}(r,g^*Z)\le_{\rm exc}\varepsilon T_f(r,H),
\]
which proves~(ii).
\end{proof}

\section{Proof of the Main Theorem}

The proof of the Main Theorem relies on estimate~\eqref{eq:tangcount}.
To establish this estimate, we decompose
\(X\cap J_{p,1}(D)\) into its irreducible components. We show that each component either has nondense image under \(\pi:X\to Y=X/B\), has image under \(\pi_1\) contained in the
codimension-at-least-two subset \(\Sigma_B(D)\subset A\), or has codimension at least two in \(X\). The first case is treated by the argument of Lemma~\ref{lem:nondominant}, while the latter two cases are treated by the above GCD-type theorem. The key ingredient is Proposition~\ref{prop:dominant}, whose proof follows the strategy of Noguchi--Winkelmann--Yamanoi; see Lemma~6.1 of \cite{NWY2008}.

Let
\[
j_{p,1}(f)\colon \mathbb C^p\longrightarrow J_{p,1}A = A\times \mathbb C^{p\dim A}
\]
be the jet lifting of \(f\). Define $\pi_1 :J_{p,1}A \rightarrow A$ be the projection.

Set
$
X=\overline{j_{p,1}(f)(\mathbb C^p)}^{\,\mathrm{Zar}}
\subset J_{p,1}A.
$ Let \(D\subset A\) be an irreducible reduced divisor. If
\(s=0\) is a reduced local equation of \(D\), then the 1-jet
subscheme \(J_{p,1}(D)\subset J_{p,1}A\) is locally defined by
\[
\mathcal I_{J_{p,1}(D)}
=
\bigl(
s,\,
d_{\bar u}s
\ \big|\ 
\bar u\in\mathcal W_{p,\leq 1}
\bigr).
\]
Equivalently, since the words in \(\mathcal W_{p,\leq1}\) are
\(1,\ldots,p\),
\begin{equation}\label{eq: -first-jet-D}\nonumber
J_{p,1}(D)
=
V\bigl(s,d_1s,\ldots,d_ps\bigr)
\subset J_{p,1}A.
\end{equation}
Here, by definition,
\begin{equation}\label{eq:jet-coordinate-pullback}\nonumber
(d_js)\bigl(j_{p,1}(f)(z)\bigr)
=
\frac{\partial(s\circ f)}{\partial z_j}(z),
\qquad 1\leq j\leq p.
\end{equation}

\begin{lemma}
\label{lem:exact-tangency-valuation}
Let \(E\) be a prime component of the divisor \(f^*D\), and let
$
m=\nu_{E}(f^*D).
$
Then
\[
\nu_{E}\bigl(j_{p,1}(f)^*J_{p,1}(D)\bigr)=m-1.
\]
\end{lemma}

\begin{proof}
Let \(g=0\) be a reduced local defining equation of \(E\). Since
$m=\nu_{E}(s\circ f),$ we may write
$s\circ f=u\,g^m,$
where \(u\) is a nowhere vanishing holomorphic function.

For every \(1\leq j\leq p\), we have
$
\frac{\partial(s\circ f)}{\partial z_j}
=
g^{m-1}
\left(
m u\frac{\partial g}{\partial z_j}
+
g\frac{\partial u}{\partial z_j}
\right).
$
Consequently,
$
\nu_{E}\left(
\frac{\partial(s\circ f)}{\partial z_j}
\right)
\geq m-1
$
for every \(j\). 
Since \(g=0\) is a reduced local defining equation of the prime
divisor \(E\), its differential does not vanish at the generic point of
\(E\). Hence there exists \(j_0\in\{1,\ldots,p\}\) such that
$\frac{\partial g}{\partial z_{j_0}}\not\equiv0\pmod{(g)}.$
Because \(u\) is nowhere vanishing, it follows that
$m u\frac{\partial g}{\partial z_{j_0}}+g\frac{\partial u}{\partial z_{j_0}}
\not\equiv0\pmod{(g)}$.
Thus
$
\nu_{E}\left(
\frac{\partial(s\circ f)}{\partial z_{j_0}}
\right)
=m-1,
$ which implies 
\[
\nu_{E}\bigl(j_{p,1}(f)^*J_{p,1}(D)\bigr)=m-1.
\]
\end{proof}

Put \(B=\St_A^0(X)\), let \(\pi:X\to Y=X/B\). Define the closed bad base locus
\[
 \Sigma_B(D)= D_{\sing}\cup\overline{\{x\in D_{\reg}~|~ \Lie B\subset T_xD\}}.
\]
If \(B\not\subset\St(D)\), $\Sigma_B(D)$ is a proper subvariety of \(D\)
and hence has codimension at least two in \(A\).  Indeed, if every
invariant vector field induced by \(\Lie(B)\) were tangent to \(D\) at a
general point, its local flows would therefore preserve \(D\). It follows that
\(B\subset\St(D)\), a contradiction.

\begin{proposition}[Tangency versus stabilizer] \label{prop:dominant}
Let \(Z\) be an irreducible component of
$ X\cap J_{p,1}{(D)}$.

(i) If \(B\subset\St(D)\), then
$$
 \overline{\pi(Z)}\subsetneq Y.
$$

(ii) If \(B\not\subset\St(D)\) and \(\pi_1(Z)\not\subset\Sigma_B(D)\), then either
$$
 \overline{\pi(Z)}\subsetneq Y ~~or~~\codim_XZ\geq2.
$$

\end{proposition}

\begin{proof} We first prove (i).
Since \(B\subset\St(D)\), 
\(J_{p,1}{(D)}\) is \(B\)-invariant and \(B\)
acts trivially on the finite set of irreducible components of
$X\cap J_{p,1}{(D)},$
so \(Z\) is \(B\)-invariant.  If \(\overline{\pi(Z)}=Y\), then \(Z/B=Y\) and
hence \(Z=X\).  It follows that
$ X\subset J_{p,1}{(D)}\subset\pi_1^{-1}(D),$
so \(f(\mathbb C^p)\subset D\), contrary to the Zariski density of \(f\).
Therefore \(\overline{\pi(Z)}\subsetneq Y\).

To prove (ii),  assume that \(B\not\subset\St(D)\) and 
\(\pi_1(Z)\not\subset\Sigma_B(D)\). Assume that
\begin{equation}\label{eq:dominant-contradiction}
 \overline{\pi(Z)}=Y,
 \qquad \codim_X Z=1.
\end{equation}
Because \(\pi|_Z\) is dominant, \(Z\cap\pi^{-1}(Y_{\reg})\) is
dense and open in \(Z\).  Since \(\pi\) is smooth, \(X\) is smooth
at every point over \(Y_{\reg}\).  Moreover, \(Z_{\reg}\) is dense
in \(Z\), and the assumption on \(\pi_1(Z)\) implies that
$ Z\cap\pi_1^{-1} \bigl(D_{\reg}\setminus\Sigma_B(D)\bigr)$
is a nonempty open subset of \(Z\).  Consequently
\[
 Z^\circ=Z_{\reg}\cap\pi^{-1}(Y_{\reg})\cap
 \pi_1^{-1}\bigl(D_{\reg}\setminus\Sigma_B(D)\bigr)
\]
is a nonempty Zariski open subset of \(Z\), and both \(X\) and \(Z\)
are smooth at every point of \(Z^\circ\).

Consider the morphism
\[
 \mu:B\times Z\longrightarrow X,\qquad(b,\eta)\longmapsto b+\eta.
\]
We have
$ \mu(B\times Z)=\pi^{-1}\bigl(\pi(Z)\bigr).$
By \eqref{eq:dominant-contradiction}, \(\pi(Z)\) is a dense
constructible subset of \(Y\) and contains a nonempty Zariski open subset in $Y$.
The same is true of \(\pi(Z^\circ)\), because \(Z^\circ\) is dense
in \(Z\).  Thus \(\mu|_{B\times Z^\circ}\) is dominant, so \(\mu(B\times Z^\circ)\) is dense constructible
in \(X\), and therefore contains a nonempty Zariski open subset
\(X^\circ\subset X\).

Since \(j_{p,1}f(\mathbb C^p)\cap X^\circ\not =\emptyset\), choose
$ z_0\in\mathbb C^p, ~b_0\in B,~\zeta=(q,w_1,\ldots,w_p)\in Z^\circ$
such that \(j_{p,1}f(z_0)=b_0+\zeta\), and define
$\widehat f(z)=-b_0+f(z+z_0)$.

Therefore
\begin{equation}\label{eq:translated- -jet}
 j_{p,1}\widehat f(z)=-b_0+j_{p,1}f(z+z_0),
 \qquad j_{p,1}\widehat f(0)=\zeta.
\end{equation}
Since \(b_0\in B\), the map \(j_{p,1}\widehat f\) still takes values in \(X\) and has
Zariski closure \(X\).  The map \(\widehat f\) is also
Zariski dense in \(A\).

Because \(q\notin\Sigma_B(D)\), there is \(v\in\Lie(B)\) such that
\begin{equation}\label{eq:transverse-v}
 v\notin T_qD.
\end{equation}
Denote $v_X(\zeta)=\left.\frac{d}{dt}\right|_{t=0}\bigl(\exp_B(tv)+\zeta\bigr)\in T_\zeta X$. Since \(\pi_1(Z)\subset D\),
$
 d{\pi_1}_\zeta(T_\zeta Z)\subset T_qD,~
 d{\pi_1}_\zeta\bigl(v_X(\zeta)\bigr)=v.
$
It follows from \eqref{eq:transverse-v} that
\begin{equation}\label{eq:action-transverse-Z}
 v_X(\zeta)\notin T_\zeta Z.
\end{equation}
By \eqref{eq:dominant-contradiction}, \(T_\zeta Z\) is a hyperplane
in \(T_\zeta X\).

Choose a holomorphic embedded disc
$
 \gamma:\Delta\longrightarrow B,~
 \gamma(0)=0, ~\gamma'(0)=v.
$
Let \(U\subset D_{\reg}\) be a sufficiently small analytic
neighborhood of \(q\). Define $ G:\Delta\times U\longrightarrow A
$ by $G(t,x)=\gamma(t)+x$. By (\ref{eq:transverse-v}) and \(\dim D=\dim A-1\), 
$dG_{(0,q)}$ is an isomorphism.  After shrinking \(\Delta\) and \(U\),
we have a biholomorphism
\begin{equation}\label{eq:base-product-chart}
 G:\Delta_0\times U\xrightarrow{\sim}\mathcal U
\end{equation}
onto an analytic neighborhood \(\mathcal U\) of \(q\).  

Define
$ \tau=\operatorname{pr}_{\Delta_0}\circ G^{-1}: \mathcal U\longrightarrow\Delta_0.$
For fixed \(t\in\Delta_0\), then
\begin{equation}\label{eq:tau-leaf}
 \tau^{-1}(t)=G(\{t\}\times U)=\gamma(t)+U.
\end{equation}

Define
$ F:\Delta\times Z\longrightarrow X$ by $ F(t,\eta)=\gamma(t)+\eta.$
By \eqref{eq:action-transverse-Z} and \(\dim Z=\dim X-1\),
$dF_{(0,\zeta)}$ is an isomorphism.  Shrinking the disc in
\eqref{eq:base-product-chart}, if necessary, and choosing an analytic
neighborhood \(Z_0\subset Z^\circ\) of \(\zeta\), with
\(\pi_1(Z_0)\subset U\), we obtain a
biholomorphism
$F:\Delta_0\times Z_0\xrightarrow{\sim}\Omega$
onto an analytic neighborhood \(\Omega\) of \(\zeta\) in \(X\).
It is easy to see that
$
 \pi_1\circ F=G\circ(\id_{\Delta_0}\times\pi_1)
  \hbox{ on }\Delta_0\times Z_0.
$
In particular, \(\pi_1(\Omega)\subset\mathcal U\).

Let
$
 (x,\alpha_1,\ldots,\alpha_p)
 =F\bigl(t,(q',\beta_1,\ldots,\beta_p)\bigr)\in\Omega,
$ where \(t\in\Delta_0\), \((q',\beta_1,\ldots,\beta_p)\in Z_0\), \(x\in\mathcal U\), and \(\alpha_j,\beta_j\in\Lie(A)\) for \(1\leq j\leq p\).
Since \(Z_0\subset J_{p,1}{(D)}\) and
\(q'\in U\subset D_{\reg}\),
$
 \beta_j\in T_{q'}D,~
  1\leq j\leq p.
$
Every jet in \(\Omega\) satisfies
\begin{equation}\label{eq:leafannihilate}
 d\tau_x(\alpha_j)=0,
 \qquad j=1,\ldots,p.
\end{equation}

By \eqref{eq:translated- -jet} and continuity, there is a
connected analytic neighborhood \(W\subset\mathbb C^p\) of \(0\) such that
$
 j_{p,1}\widehat f(W)\subset\Omega.
$
By \eqref{eq:leafannihilate}, we get
$
 \frac{\partial}{\partial z_j}(\tau\circ\widehat f)(z)
 =0,
  1\leq j\leq p, z\in W.
$
Thus \(\tau\circ\widehat f\) is constant on \(W\).  If its value is
\(t_0\), then \eqref{eq:tau-leaf} gives
$
 \widehat f(W)\subset\gamma(t_0)+U
 \subset\gamma(t_0)+D.
$
$
 \widehat f^{-1}\bigl(\gamma(t_0)+D\bigr)
$
is a closed analytic subset of the connected complex manifold
\(\mathbb C^p\), and it contains the nonempty open set \(W\).  It is
therefore all of \(\mathbb C^p\).  Hence, the image of
\(\widehat f\) is contained in the proper algebraic subset
\(\gamma(t_0)+D\), contradicting its Zariski density.
\end{proof}

\begin{theorem}\label{prop:tangcount}
For every irreducible reduced divisor \(D\subset A\) and every
\(\eps>0\),
\begin{equation}\label{eq:tangcount}
 N^{[1]}
 \left(r,(j_{p,1}f)^*J_{p,1}{(D)}\right)
 \leq_{\mathrm{exc}} \eps T_f(r,H).
\end{equation}
\end{theorem}

\begin{proof}
We have $ N^{[1]}
 \left(r,(j_{p,1}f)^*J_{p,1}{(D)}\right)= N^{[1]}
 \left(r,(j_{p,1}f)^*(J_{p,1}{(D)\cap X)}\right)$.
Decompose \( J_{p,1}{(D)\cap X}\) into finitely many irreducible
components \(Z_\alpha\).

If \(B\subset\St(D)\), Proposition~\ref{prop:dominant} says that
every \(Z_\alpha\) has nondense image under
\(\pi:X\to X/B\).  Applying the argument of
Lemma~\ref{lem:nondominant} to \(\pi\circ j_{p,1}f\), we get \eqref{eq:tangcount}. This finishes the
proof in this case.

Assume that \(B\not\subset\St(D)\), so
\(\codim_A\Sigma_B(D)\geq2\).
If the base projection of \(Z_\alpha\) is contained in
\(\Sigma_B(D)\), applying Theorem~\ref{cor:HC-applications} to $f$ and \(\Sigma_B(D)\), we obtain the required estimate.

For the other components, Proposition~\ref{prop:dominant} applies.
If $\overline{\pi(Z_\alpha)}\subsetneq Y$, use the same argument in Lemma~\ref{lem:nondominant} for \(\pi\circ j_{p,1}f\).
If it is dense, then $\codim_XZ_\alpha\geq2$.
Apply Theorem~\ref{cor:HC-applications} to\ \(j_{p,1}f\) and $Z_\alpha$. This finishes the
proof.
\end{proof}

\begin{proof}[Proof of Theorem~\ref{thm:main}]
After quotienting by \(\operatorname{St}^0(D)\), we may assume that
\(\operatorname{St}^0(D)=\{0\}.\)
Hence  \(\overline D\) is big, which implies $T_f(r,\overline D)\asymp T_f(r,H)$.
Write $D=\sum_{i=1}^qD_i$ where \(D_1,\ldots,D_q\) are the distinct prime components of $D$. By Lemma \ref{lem:exact-tangency-valuation}, one has
\begin{align}\nonumber
 N_f^{[\rho]}(r,D)
 \leq {}&
 N_f^{[1]}(r,D)
 +(\rho-1)\sum_{i=1}^q
  N^{[1]}
 \left(r,(j_{p,1}f)^*J_{p,1}{(D_i)}\right)\notag\\\nonumber
 &+\sum_{1\leq i<j\leq q}
 N^{[1]}(r,f^*(D_i\cap D_j)).\nonumber
\label{eq:countingcompare}
\end{align}
Since $\codim_A (D_i\cap D_j) \geq2$ for $i\not=j$, Theorem \ref{thm:main} follows from Theorem \ref{SEMIABELIAN}, Theorem \ref{cor:HC-applications} (i) and Theorem \ref{prop:tangcount}.
\end{proof}

\noindent\textbf{Acknowledgments.} 
The author would like to thank his advisor, Min Ru, for the constant support and encouragement, as well as for many valuable discussions and suggestions.


\begin{thebibliography}{10}
  

\bibitem{Bloch1926}
A.~Bloch,
\emph{Sur les syst\`emes de fonctions uniformes satisfaisant \`a l'\'equation d'une vari\'et\'e alg\'ebrique dont l'irr\'egularit\'e d\'epasse la dimension},
J. Math. Pures Appl. \textbf{5} (1926), 19--66.

\bibitem{CaiRuYang2}
Q.~Cai, M.~Ru and C. J.~Yang,
\emph{Defect relations for meromorphic maps of complete K\"ahler manifolds into projective varieties},
Complex Anal. Synerg. \textbf{11} (2025), article 14.

\bibitem{CDY25}
B.~Cadorel, Y.~Deng and K.~Yamanoi,
\emph{Hyperbolicity and fundamental groups of complex quasi-projective varieties (I):
Maximal quasi-Albanese dimension by Nevanlinna theory},
arXiv:2511.04405, 2025.

\bibitem{Dem}
J.-P.~Demailly,
\emph{Algebraic criteria for Kobayashi hyperbolic varieties and jet differentials},
Proc. Sympos. Pure Math. \textbf{62}, Part~2 (1995), 285--360.

\bibitem{DL01}
G.-E.~Dethloff and S.~S.-Y.~Lu,
\emph{Logarithmic jet bundles and applications},
Osaka J. Math. \textbf{38} (2001), no.~1, 185--237.

\bibitem{Ete}
A.~Etesse,
\emph{Geometric generalized Wronskians: Applications in intermediate hyperbolicity and foliation theory},
Int. Math. Res. Not. IMRN \textbf{2023} (2023), no.~10, 8251--8310.

\bibitem{AG}
R.~Hartshorne,
\emph{Algebraic Geometry},
Graduate Texts in Mathematics, vol.~52,
Springer-Verlag, New York--Heidelberg, 1977.



\bibitem{Noguchi1977}
J.~Noguchi,
\emph{Holomorphic curves in algebraic varieties},
Hiroshima Math. J. \textbf{7} (1977), 833--853.

\bibitem{Noguchi1981}
J.~Noguchi,
\emph{Lemma on logarithmic derivatives and holomorphic curves in algebraic varieties},
Nagoya Math. J. \textbf{83} (1981), 213--233.

\bibitem{Noguchi1998}
J.~Noguchi,
\emph{On holomorphic curves in semi-abelian varieties},
Math. Z. \textbf{228} (1998), 713--721.

\bibitem{NO}
J.~Noguchi and T.~Ochiai,
\emph{Geometric Function Theory in Several Complex Variables},
Japanese edition, Iwanami, Tokyo, 1984;
English translation, Transl. Math. Monogr., vol.~80,
American Mathematical Society, Providence, RI, 1990.

\bibitem{NW03}
J.~Noguchi and J.~Winkelmann,
\emph{A note on jets of entire curves in semi-abelian varieties},
Math. Z. \textbf{244} (2003), 705--710.

\bibitem{NW04}
J.~Noguchi and J.~Winkelmann,
\emph{Bounds for curves in abelian varieties},
J. Reine Angew. Math. \textbf{572} (2004), 27--47.

\bibitem{NWbook}
J.~Noguchi and J.~Winkelmann,
\emph{Nevanlinna Theory in Several Complex Variables and Diophantine Approximation},
Grundlehren der mathematischen Wissenschaften, vol.~350,
Springer, Tokyo, 2014.

\bibitem{NWY2002}
J.~Noguchi, J.~Winkelmann and K.~Yamanoi,
\emph{The second main theorem for holomorphic curves into semi-abelian varieties},
Acta Math. \textbf{188} (2002), no.~1, 129--161.

\bibitem{NWY2008}
J.~Noguchi, J.~Winkelmann and K.~Yamanoi,
\emph{The second main theorem for holomorphic curves into semi-abelian varieties II},
Forum Math. \textbf{20} (2008), no.~3, 469--503.

\bibitem{Ochiai1977}
T.~Ochiai,
\emph{On holomorphic curves in algebraic varieties with ample irregularity},
Invent. Math. \textbf{43} (1977), 83--96.

\bibitem{PR}
G.~Pacienza and E.~Rousseau,
\emph{Generalized Demailly--Semple jet bundles and holomorphic mappings into complex manifolds},
J. Math. Pures Appl. (9) \textbf{96} (2011), no.~2, 109--134.

\bibitem{ru_annals}
M.~Ru,
\emph{Holomorphic curves into algebraic varieties},
Ann. of Math. (2) \textbf{169} (2009), no.~1, 255--267.

\bibitem{book}
M.~Ru,
\emph{Nevanlinna Theory and Its Relation to Diophantine Approximation},
2nd ed.,
World Scientific, Singapore, 2021.

\bibitem{HCG}
Y.-T.~Siu,
\emph{Hyperbolicity in complex geometry},
in \emph{The Legacy of Niels Henrik Abel},
Springer, Berlin, 2004, 543--566.

\bibitem{SiuYeung1996}
Y.-T.~Siu and S.-K.~Yeung,
\emph{A generalized Bloch's theorem and the hyperbolicity of the complement of an ample divisor in an Abelian variety},
Math. Ann. \textbf{306} (1996), 743--758.

\bibitem{SiuYeung1997}
Y.-T.~Siu and S.-K.~Yeung,
\emph{Defects for ample divisors of abelian varieties, Schwarz lemma, and hyperbolic hypersurfaces of low degree},
Amer. J. Math. \textbf{119} (1997), 1139--1172.

\bibitem{Stoll1953}
W.~Stoll,
\emph{Die beiden Haupts\"atze der Wertverteilungstheorie. I},
Acta Math. \textbf{90} (1953), 1--115.

\bibitem{Stoll}
W.~Stoll,
\emph{Holomorphic Functions of Finite Order in Several Complex Variables},
CBMS Regional Conference Series in Mathematics, vol.~21,
American Mathematical Society, Providence, RI, 1974.

\bibitem{Stoll1977}
W.~Stoll,
\emph{Value Distribution on Parabolic Spaces},
Lecture Notes in Mathematics, vol.~600,
Springer-Verlag, Berlin--Heidelberg--New York, 1977.

\bibitem{Vitter1977}
A.~L.~Vitter,
\emph{The lemma of the logarithmic derivative in several complex variables},
Duke Math. J. \textbf{44} (1977), no.~1, 89--104.

\bibitem{WangCp}
Z.~Wang,
\emph{Meromorphic maps from $\mathbb C^p$ into semi-abelian varieties and general projective varieties},
J. Geom. Anal. \textbf{36} (2026), article 260,
doi:10.1007/s12220-026-02508-8.

\bibitem{Wong1980}
P.-M.~Wong,
\emph{Holomorphic mappings into abelian varieties},
Amer. J. Math. \textbf{102} (1980), no.~3, 493--502.

\bibitem{Y04}
K.~Yamanoi,
\emph{Holomorphic curves in abelian varieties and intersections with higher codimensional subvarieties},
Forum Math. \textbf{16} (2004), 749--788.

\bibitem{Yamanoi2015}
K.~Yamanoi,
\emph{Holomorphic curves in algebraic varieties of maximal Albanese dimension},
Int. J. Math. \textbf{26} (2015), no.~6,
1541006, 45 pp.

\bibitem{Yamanoi2}
K.~Yamanoi,
\emph{Kobayashi hyperbolicity and higher-dimensional Nevanlinna theory},
in \emph{Geometry and Analysis on Manifolds},
Progress in Mathematics, vol.~308,
Birkh\"auser/Springer, 2015, 209--273.

\end{thebibliography}
\end{document}